\documentclass[reqno,12pt]{amsart}
\usepackage[left=1.2in, right = 1.2in, top=1in]{geometry}
\usepackage{amsmath, amssymb}
\usepackage{amsthm}
\usepackage{tikz}
\usepackage{amstext}
\usepackage{pifont}
\usepackage{appendix}
\usepackage{verbatim}
\usepackage{cases}
\newtheorem{theorem}{Theorem}[section]
\newtheorem{remark}{Remark}[section]

\newtheorem{lemma}{Lemma}[section]

\newtheorem{claim}{claim}

\usepackage{fancyhdr}
\allowdisplaybreaks[4]
\usepackage{hyperref}

\title{Liouville theorem for elliptic equations involving the sum of the function and its gradient in $\mathbb R^n$}

\author{Xi-Nan Ma}
\address{Department of Mathematics \\University of Science and Technology of China, Hefei, China}
\email{xinan@ustc.edu.cn}
\author{Wangzhe Wu}
\address{Institute of Mathematics \\Academy of Mathematics and Systems Science, Chinese
Academy of Sciences, Beijing, 100190, China}
\email{wuwz18@mail.ustc.edu.cn}

\author{Qiqi Zhang}
\address{College of Mathematics and Information Science \\ Nanchang Hangkong University, Nanchang, China}
\email{qiqizmath@126.com}

\begin{document}
	
	\pagestyle{fancy}

\fancyhead{}
\fancyhead[CO]{Liouville theorem for elliptic equations}

\fancyhead[CE]{\leftmark}

	\begin{abstract}
		We prove Liouville theorem for the equation $\Delta v + N v^p + M |\nabla v|^{q}=  0$ in $\mathbb R^n$, with $M, N > 0, q = \frac{2p}{p + 1}$ in the critical and subcritical case. The proof is based on a differential identity and Young inequality.
	\end{abstract}

	\maketitle

	\section{Introduction}
\noindent In this paper, we are concerned with a Liouville type theorem for the global solutions in the critical and subcritical cases of the following equations in $\mathbb{R}^n$:
	\begin{equation}\label{mainresult}
		\Delta v + N v^p + M |\nabla v|^{q}=0,
	\end{equation}
where $q = \frac{2p}{p + 1}$,  $p$ and $q$ are exponents larger than 1 and $M,  N > 0$. This kind of equations are widely concerned and studied depending on the different value of $M, N$.
	
For the case $M = 0, N = 1$, the work concerning the symmetry of solutions of second order elliptic equations on an unbounded domain was first done by Gidas, Ni and Nirenberg \cite{MR634248}, and then generalized to infinite cylinders by Berestycki and Nirenberg \cite{MR1039342}. Gidas, Ni and Nirenberg \cite{MR634248}  proved that for $p = \frac{n + 2}{n - 2}$, all the positive solutions with reasonable behavior at infinity, namely $v = O(|x|^{2 - n})$, are radially symmetric about some point and have the form
$$
u = \Big[\frac{\lambda \sqrt{n(n - 2)}}{\lambda^2 + |x - x_0|^2} \Big]^{\frac{n - 2}{2}}.
$$
This uniqueness result, as was pointed out by R. Schoen, is in fact equivalent to the geometric result due to Obata \cite{MR303464}. In the case that $1 < p < \frac{n + 2}{n - 2}$, Gidas and Spruck \cite{MR615628} showed that the only nonnegative solution is trivial.
The related equation
$$
\Delta v+v^p=0,
$$
is usually called the Lane-Emden equation. It is well-known that radial ground states exist if and only if the exponent $p$ either is critical (i.e.=$(n+2)/(n-2)$) or is supercritical ($>(n+2)/(n-2)$). A fairly straightforward modern proof of Fowler's result can be given using standard shooting methods together with generalized Poho$\check{z}$aev identities.

  Due to Caffarelli, Gidas and Spruck \cite{MR982351}, the treatment of the critical case $p = \frac{n + 2}{n - 2}$ was made possible thanks to a completely new approach based upon a combination of moving plane analysis and geometric measure theory. 
Later, Chen-Li \cite{MR1121147} provided a simpler proof than \cite{MR982351}. Recently, Catino and Monticelli \cite{catino} get some classification results on Riemannian manifolds.  For more application of the differential identity and integration by parts, one can refer to \cite{MR4753063, MR4519639}.

For the case $M = 1, N = 0$, it is evident that equation \eqref{mainresult} becomes the Hamilton-Jacobi equation
$$
\Delta v  +  |\nabla v|^{q}=0,
$$
which is proved by P.L. Lions \cite{MR833413} that any $C^2$ solution in $\mathbb R^n$ has to be a constant for $q > 1$.  Other work on this kind of equation can be seen in \cite{MR3261111}.
Peculiarly, it can be also seen as another special case of the below equation which can be expressed as the product of $v^{p}$ and $|\nabla v|^q$:
	\begin{equation}\label{Veronequ1}
		\begin{aligned}
			-\Delta v = v^{p}|\nabla v|^q, ~\text{in}\quad  \mathbb R^n.
		\end{aligned}
	\end{equation}
Such equation was studied  by Bidaut-V\'eron,  Garc\'ia-Huidobro,
and  V\'eron\cite{MR3959864}. They obtained Liouville type theorem with the conditions on $p, q$. They also studied local and global properties of positive solutions in the range $p+q > 1$, $p \ge 0$ and $0 \le q < 2$. Their main results dealt with the subcritical range and proved a priori estimates of positive solutions of \eqref{Veronequ1} in a punctured domain and existence of ground states in $\mathbb{R}^n$, based on two
approaches for obtaining their results: the direct Bernstein method by Lions \cite{MR588033} and the integral Bernstein method popularized by Gidas and Spruck in \cite{MR615628} respectively. Both methods are based upon differentiating the equation. For the $p$-Laplacian case, one can refer to \cite{MR4234084}. And the group of Youde Wang generalized the results from $\mathbb R^n$ to Riemannian manifolds \cite{ he2024nashmoseriterationapproachlogarithmic, he2024optimalliouvilletheoremslaneemden, MR4703505, MR4559367}.

In \cite{MR1000727} and \cite{MR1188490} part of the study deals with the case $q \neq \frac{2p}{p + 1}$. In the critical case,  not only the sign of $M$ but also its value plays a fundamental role, with a delicate interaction with the exponent $p$. Other work on this kind of equation can be seen in \cite{MR3261111}.

Actually, there have already been some results of the same type equation as (\ref{mainresult}) in \cite{MR4150912} and \cite{MR4043662} where it is concerned with local and global properties of positive solutions with $N=1$ in a domain $\Omega\subset \mathbb R^n$, in the range $\min\{p, q\} > 1$, and $M\in \mathbb{ R}$. Particularly, Bidaut-V\'eron,  Garc\'ia-Huidobro,
and  V\'eron \cite{MR4150912} pointed out that there exists no nontrivial nonnegative solution with $n\ge 1$, $p, q > 1$, $q \neq \frac{2p}{p + 1}$ and some extra conditions on $u$.
In fact, one (see for example \cite{MR4150912}) observes that the equation \eqref{mainresult} is invariant under the scaling transformation $T_k$ with $k > 0$ if and only if $q$ is critical with respect to $p$, i.e. $q = \frac{2p}{p + 1}$ where
	\begin{equation}
		T_k[v] (x) = k^{\frac{2}{p - 1}}v(k x).
	\end{equation}
It follows that by rescaling, we can always assume $N = 1$ for our case. In the critical case, first studies in the case $M < 0$ are due to Chipot and Weissler \cite{MR1000727} for $n = 1$. The case $ n\ge 2$ was left open by Serrin and Zou \cite{MR1188490} who performed a very detailed analysis and the first partial results are due to Fila and Quittner \cite{MR1239247} and Voirol \cite{MR1422294}. Much less is known for $M > 0$. References \cite{MR4150912} and \cite{MR4043662} have done some work and also gave a Liouville type theorem.  We mainly cited some related results from Bidaut-V\'eron,  Garc\'ia-Huidobro,
and  V\'eron \cite{MR4150912}.
 	\begin{theorem}[Theorem C in \cite{MR4150912}]
		Let $n \geq 1$, $p > 1 , q = \frac{2p}{p + 1}$. For any
		\begin{align*}
			M > \left(\frac{p - 1}{p + 1}\right)^{\frac{p - 1}{p + 1}}\left(\frac{n(p + 1)^2}{4p}\right)^{\frac{p}{p + 1}},
		\end{align*}
		there exists no nontrivial nonnegative solution of \eqref{mainresult} in $\mathbb R^n$.
	\end{theorem}
	\begin{theorem}[Theorem D in \cite{MR4150912}]
		Let $n \geq 2$, $1 < p < \frac{n + 3}{n - 1}, q = \frac{2p}{p + 1}$. For any $M > 0$,
		there exists no nontrivial nonnegative solution of \eqref{mainresult} in $\mathbb R^n$.
	\end{theorem}
 	\begin{theorem}[Theorem E in \cite{MR4150912}] \label{Veron0}
		Let $n \geq 3$, $1 < p < \frac{n + 2}{n - 2} , q = \frac{2p}{p + 1}$. Then there exists $\epsilon_0 > 0$ depending on $n$ and $p$ such that for any $|M| \leq \epsilon_0$,
		there exists no nontrivial nonnegative solution of \eqref{mainresult} in $\mathbb R^n$.
	\end{theorem}

By observation, we found that the results in \cite{MR4150912} have not covered all the cases that $1 < p \leq \frac{n + 2}{n - 2}$ for $q = \frac{2p}{p + 1}, M > 0$.
Motivated by the results above, we aim to complete the corresponding results of the range of $p$ with $1 < p \le \frac{n + 2}{n - 2}, M > 0$, based on a
differential identity and Young inequality. 
Our main results in this framework are the following and we always assume that $N = 1$:
	\begin{theorem}\label{main result1}
		Let $n \geq 3$, $1 < p \leq \frac{n + 2}{n - 2}$, $q = \frac{2p}{p + 1}$, then all the nonnegative solutions of \eqref{mainresult} are $v \equiv 0$ for any $M > 0$.
	\end{theorem}
	In order to prove Theorem \ref{main result1}, we need two other conclusions:
	\begin{theorem}\label{Theorem1}
		When $1 < p \leq \frac{n + 2}{n - 2}$, $q = \frac{2p}{p + 1}$, $n \geq 3$, there exists $M_1 > 0$ such that if $0 < M < M_1$, then $v \equiv 0$. Especially, if $3 \leq n \leq 6$, then $M_1 = \infty$.
	\end{theorem}
Here, $M_1$ will be determined later in Section 4. For $3\le n\le 6$, by simple computation, we can prove that for any $M$, Theorem \ref{Theorem1} holds. But for $n\ge 7$, we need to apply Young inequality.	
	\begin{theorem}\label{Theorem2}
		For any $n \geq 3, 1 < p \leq \frac{n + 2}{n - 2}$, $q = \frac{2p}{p + 1}$, there exists $M_2 > 0$, such that if $M > M_2$, then $v \equiv 0$.
	\end{theorem}
	Besides, after computation, we have the following lemma:
	
	\begin{lemma}\label{Lem0}
		For any $1 < p \leq \frac{n + 2}{n - 2}$, $q = \frac{2p}{p + 1}$, we have $M_2 < M_1$.
	\end{lemma}	

 One new idea in this paper is that we introduce the new parameter $S,Q$ in (\ref{0.1}), but before our work one always chooses $S=Q=\frac{1}{n}$ as in  (\ref{0.2}). This new idea obtains one more freedom to get the positive term in (\ref{TheFinalEqua1}).

For convenience, we make a summary of ideas and article planning arrangement as follows. We expand the range of $p$ of the result in \cite{MR4150912} to $1 < p \le \frac{n + 2}{n - 2}$ for Liouville type theorem of the equation \eqref{mainresult}, that is we not only prove the subcritical case but also the critical case which calls for much more complicated computations. Besides, contrast to Theorem \ref{Veron0} and Theorem \ref{main result1}, we even remove the sufficient small condition for constants $M$ in \cite{MR4150912}. The main method taken in our work is based on a differential identity and Young inequality. This paper is organized as follows: we use integration by parts to get the differential identity in Section 2. According to the differential identity, some conditions are given to make the nonnegative solution is zero in Section 3. On the basis of satisfying the above conditions, we introduce Young inequality and prove that there is only zero solution in Section 4. Finally, we're going to talk about the upper bound and lower bound of $M$ separately in Section 5 and Section 6, and point out that the lower bound is smaller that the upper bound. Especially, for the upper bound $M_1$, we classify the cases where $2\le n\le 6$ and $n\ge 7$, then we complete the proof of the Lemma \ref{Lem0}. In conclusion, it follows that there is no requirement of positive constant $M$ to make Liouville type theorem hold.

	{\it Acknowledgement.} The authors thank Prof. Bidaut-V\'eron 
	and  V\'eron for their very useful comments on this paper.
	The authors were supported by  National Natural Science Foundation of China (grants No.12141105 and grants No.12301098) and National Key Research and Development Project (grants No.SQ2020YFA070080).
	\section{Integral identity}\label{1}
	Consider the equation:
	\begin{equation}\label{lap}
		\Delta v + N v^p + M |\nabla v|^{q}=0.
	\end{equation}
Multiplying \eqref{lap} by $v^{\alpha}|\nabla v|^{\gamma}\Delta v$,
\begin{equation}\label{2}
		\begin{aligned}
			   & v^\alpha|\nabla v|^{\gamma}(\Delta v)^2  =  -N v^{\alpha + p } |\nabla v|^{\gamma} \Delta v  - M v^{\alpha}|\nabla v| ^{\gamma + q}\Delta v. 
		\end{aligned}
	\end{equation}
	We observed that the term $v^\alpha|\nabla v|^{\gamma}(\Delta v)^2$ is important.
	In the following contents, we always choose $\gamma = 0$ or $\gamma \geq 3$. And at one fixed point $x_0$, suppose that $v_1(x_0) = |\nabla v|(x_0)$ and define
	\begin{equation}
		\begin{aligned}\label{0.1}
			&G_{11} = v_{11} - S \Delta v,\\
			&G_{ij} = v_{ij} - Q\delta_{ij}\Delta v, \quad i + j > 2,\\
		\end{aligned}
	\end{equation}
	and
	\begin{equation}\label{0.2}
		E_{ij} = v_{ij} - \frac{1}{n}\delta_{ij}\Delta v, \text{ for }  i, j = 1,...,n.
	\end{equation}
	We remark that the definition of $G_{ij}$ is depending on the choice of frame but $E_{i j}$ is not.

	\begin{itemize}
		\item The term $v^{\alpha - 1}|\nabla v|^{\gamma + 2}\Delta v $:
	
	\begin{equation}\label{section1equ5}
		\begin{aligned}
			v^{\alpha - 1}|\nabla v|^{\gamma + 2}\Delta v &= (v^{\alpha - 1}|\nabla v|^{\gamma + 2}v_i)_i - (\alpha - 1)v^{\alpha - 2}|\nabla v|^{\gamma +4} - (\gamma + 2)v^{\alpha - 1}|\nabla v|^{\gamma}v_i v_j v_{ij}\\
			&= (v^{\alpha - 1}|\nabla v|^{\gamma + 2}v_i)_i - (\alpha - 1)v^{\alpha - 2}|\nabla v|^{\gamma +4}\\
			& - (\gamma + 2)v^{\alpha - 1}|\nabla v|^{\gamma + 2}(G_{11} + S\Delta v),\\
			\Rightarrow v^{\alpha - 1}|\nabla v|^{\gamma + 2}\Delta v &= \frac{1}{1 + \gamma S + 2S} (v^{\alpha - 1}|\nabla v|^{\gamma + 2}v_i)_i - \frac{\alpha - 1}{1 + \gamma S + 2S}v^{\alpha - 2}|\nabla v|^{\gamma +4}\\
			&- \frac{\gamma + 2}{1 + \gamma S + 2S} v^{\alpha - 1}|\nabla v|^{\gamma + 2} G_{11}.
		\end{aligned}
	\end{equation}

	\item The term $v^{\alpha}|\nabla v|^{\gamma}(\Delta v)^2$: in fact, we have
	\begin{equation*}
		\begin{aligned}
			v^{\alpha}|\nabla v|^{\gamma}(\Delta v)^2 & = v^{\alpha}|\nabla v|^{\gamma}v_{ii}v_{j j}\\
			&= (v^{\alpha}|\nabla v|^\gamma v_{ii}v_j)_j - \alpha v^{\alpha - 1}|\nabla v|^{\gamma + 2}v_{ii} - v^{\alpha}|\nabla v|^{\gamma}_j v_{ii}v_j - v^\alpha |\nabla v|^\gamma v_{i i j}v_j,
		\end{aligned}
	\end{equation*}
	\begin{equation*}
		\begin{aligned}
			v^\alpha |\nabla v|^\gamma v_{iij}v_j &= (v^\alpha |\nabla v|^{\gamma }v_{ij}v_j)_i - (v^{\alpha}|\nabla v|^\gamma v_j)_iv_{ij}\\
			&= (v^\alpha |\nabla v|^{\gamma }v_{ij}v_j)_i - \alpha v^{\alpha - 1}|\nabla v|^\gamma v_i v_j v_{ij} - v^{\alpha}|\nabla v|^\gamma_iv_{ij}v_j - v^{\alpha}|\nabla v|^{\gamma}v_{ij}^2.
		\end{aligned}
	\end{equation*}
	As a result, we obtain
	\begin{equation}
		\begin{aligned}
			v^\alpha|\nabla v|^{\gamma}(\Delta v)^2 &= v^{\alpha}|\nabla v|^{\gamma}(\Delta v)^2\\
			&= (v^{\alpha}|\nabla v|^\gamma v_{ii}v_j)_j - \alpha v^{\alpha - 1}|\nabla v|^{\gamma + 2}\Delta v - v^{\alpha}|\nabla v|^{\gamma}_j \Delta v\cdot v_j \\
			&- (v^\alpha |\nabla v|^{\gamma }v_{ij}v_j)_i + \alpha v^{\alpha - 1}v_i v_j v_{ij}|\nabla v|^\gamma + v^{\alpha}|\nabla v|^\gamma_iv_{ij}v_j + v^{\alpha}|\nabla v|^{\gamma}v_{ij}^2\\
			&=  (v^{\alpha}|\nabla v|^{\gamma}\Delta v v_i -  v^{\alpha}|\nabla v|^{\gamma}v_{ij}v_j)_i + v^{\alpha}|\nabla v|^{\gamma}v_{ij}^2 - \alpha v^{\alpha - 1}|\nabla v|^{\gamma + 2}\Delta v\\
			& + \alpha v^{\alpha - 1}|\nabla v|^{\gamma}v_i v_j v_{ij} - v^{\alpha}\Delta v\cdot v_j|\nabla v|^{\gamma}_j + v^{\alpha}v_j v_{ij}|\nabla v|^{\gamma}_i,
		\end{aligned}
	\end{equation}
	
	\begin{align*}
		\Rightarrow &v^\alpha|\nabla v|^{\gamma}(\Delta v)^2 = \left(1 - \frac{1}{n}\right)\Big( v^{\alpha}|\nabla v|^{\gamma}\Delta v v_i\Big)_i - ( v^{\alpha}|\nabla v|^{\gamma}E_{ij}v_j)_i + \sum_{i + j > 2} v^{\alpha}|\nabla v|^{\gamma}v_{ij}^2\\
			&+ v^{\alpha}|\nabla v|^{\gamma}G_{11}^2 + 2S v^{\alpha}|\nabla v|^{\gamma}G_{11}\Delta v + S^2 v^{\alpha}|\nabla v|^{\gamma}(\Delta v)^2 - \alpha v^{\alpha - 1}|\nabla v|^{\gamma + 2}\Delta v\\
			& + \alpha v^{\alpha - 1}|\nabla v|^{\gamma + 2}v_{11} - \gamma v^{\alpha}|\nabla v|^{\gamma - 2}v_{i} v_{j} v_{ij}\Delta v + \gamma v^{\alpha}|\nabla v|^{\gamma - 2}v_{j}v_{ij}v_{l}v_{il}\\
			&= \left(1 - \frac{1}{n}\right)\Big( v^{\alpha}|\nabla v|^{\gamma}\Delta v v_i\Big)_i - ( v^{\alpha}|\nabla v|^{\gamma}E_{ij}v_j)_i + \sum_{i + j > 2} v^{\alpha}|\nabla v|^{\gamma}v_{ij}^2\\
			&+ v^{\alpha}|\nabla v|^{\gamma}G_{11}^2 + 2S v^{\alpha}|\nabla v|^{\gamma}G_{11}\Delta v + S^2 v^{\alpha}|\nabla v|^{\gamma}(\Delta v)^2 - \alpha v^{\alpha - 1}|\nabla v|^{\gamma + 2}\Delta v \\
			&+ \alpha v^{\alpha - 1}|\nabla v|^{\gamma + 2}v_{11} - \gamma v^{\alpha}|\nabla v|^{\gamma } v_{11}\Delta v + \gamma v^{\alpha}|\nabla v|^{\gamma }v_{i1}^2\\
			 &= \left(1 - \frac{1}{n}\right)\Big( v^{\alpha}|\nabla v|^{\gamma}\Delta v v_i\Big)_i - ( v^{\alpha}|\nabla v|^{\gamma}E_{ij}v_j)_i + \sum_{i + j > 2} v^{\alpha}|\nabla v|^{\gamma}v_{ij}^2\\
			&+ v^{\alpha}|\nabla v|^{\gamma}G_{11}^2 + 2S v^{\alpha}|\nabla v|^{\gamma}G_{11}\Delta v + S^2 v^{\alpha}|\nabla v|^{\gamma}(\Delta v)^2 - \alpha v^{\alpha - 1}|\nabla v|^{\gamma + 2}\Delta v \\
			&+ \alpha v^{\alpha - 1}|\nabla v|^{\gamma + 2}\Big(G_{11} + S\Delta v\Big) - \gamma v^{\alpha}|\nabla v|^{\gamma }\Big(G_{11} + S\Delta v\Big) \Delta v + \gamma v^{\alpha}|\nabla v|^{\gamma }\Big(G_{11} + S\Delta v\Big)^2\\
			&+ \sum_{i > 1} \gamma v^{\alpha}|\nabla v|^{\gamma}G_{1i}^2.
	\end{align*}
	Thus we get that
	\begin{align*}
		\Rightarrow & \Big(1 - S^2 + \gamma S - \gamma S^2\Big) v^{\alpha}|\nabla v|^{\gamma}(\Delta v)^2\\
			&=  \left(1 - \frac{1}{n}\right)\Big( v^{\alpha}|\nabla v|^{\gamma}\Delta v v_i\Big)_i - \Big( v^{\alpha}|\nabla v|^{\gamma}E_{ij}v_j\Big)_i + \sum_{i + j > 2} v^{\alpha}|\nabla v|^{\gamma}v_{ij}^2\\
			&+ \Big(\alpha S - \alpha\Big)v^{\alpha - 1}|\nabla v|^{\gamma + 2}\Delta v + \alpha v^{\alpha - 1}|\nabla v|^{\gamma + 2}G_{11}\\
			&+ \Big(2\gamma S + 2S - \gamma\Big) v^{\alpha} |\nabla v|^{\gamma}G_{11}\Delta v + \Big(1 + \gamma\Big) v^{\gamma}|\nabla v|^{\gamma}G_{11}^2 + \sum_{i > 1} \gamma v^{\alpha}|\nabla v|^{\gamma }G_{i1}^2 \\
			&=  \left(1 - \frac{1}{n}\right)\Big( v^{\alpha}|\nabla v|^{\gamma}\Delta v v_i\Big)_i - \Big( v^{\alpha}|\nabla v|^{\gamma}E_{ij}v_j\Big)_i + \sum_{i + j > 2} v^{\alpha}|\nabla v|^{\gamma}v_{ij}^2\\
			&+ \Big(\alpha S - \alpha\Big)\cdot \Bigg[ \frac{1}{1 + \gamma S + 2S} (v^{\alpha - 1}|\nabla v|^{\gamma + 2}v_i)_i - \frac{\alpha - 1}{1 + \gamma S + 2S}v^{\alpha - 2}|\nabla v|^{\gamma +4}\\
			&- \frac{\gamma + 2}{1 + \gamma S + 2S} v^{\alpha - 1}|\nabla v|^{\gamma + 2} G_{11}\Bigg] + \alpha v^{\alpha - 1}|\nabla v|^{\gamma + 2}G_{11}\\
			&+ \Big(2\gamma S + 2S - \gamma\Big) v^{\alpha} |\nabla v|^{\gamma}G_{11}\Delta v + \Big(1 + \gamma\Big) v^{\gamma}|\nabla v|^{\gamma}G_{11}^2 + \sum_{i > 1} \gamma v^{\alpha}|\nabla v|^{\gamma }G_{i1}^2,
	\end{align*}

	\begin{equation}\label{equ111}
		\begin{aligned}
			\Rightarrow & \Big(1 - S^2 + \gamma S - \gamma S^2\Big)v^{\alpha}|\nabla v|^{\gamma}(\Delta v)^2\\
			&=  \left(1 - \frac{1}{n}\right)\Big( v^{\alpha}|\nabla v|^{\gamma}\Delta v v_i\Big)_i - \Big( v^{\alpha}|\nabla v|^{\gamma}E_{ij}v_j\Big)_i + \frac{\alpha S - \alpha}{1 + \gamma S + 2S} \Big(v^{\alpha - 1}|\nabla v|^{\gamma + 2}v_i\Big)_i\\
			&+ \sum_{i + j > 2} v^{\alpha}|\nabla v|^{\gamma}v_{ij}^2  + \frac{\alpha(\alpha - 1)(1 - S)}{1 + \gamma S + 2S}v^{\alpha - 2}|\nabla v|^{\gamma +4}+  \frac{\alpha(\gamma + 3)}{1 + \gamma S + 2S} v^{\alpha - 1}|\nabla v|^{\gamma + 2} G_{11}\\
			&+ \Big(2\gamma S - \gamma + 2S\Big) v^{\alpha} |\nabla v|^{\gamma}G_{11}\Delta v + (1 + \gamma) v^{\gamma}|\nabla v|^{\gamma}G_{11}^2 + \sum_{i > 1} \gamma v^{\alpha}|\nabla v|^{\gamma }G_{i1}^2 .\\
		\end{aligned}
	\end{equation}	
	Recall that $G_{ij} := v_{ij} - Q\delta_{ij}\Delta v  , i + j > 2$, then
	\begin{equation*}
		\begin{aligned}
			\sum_{i > 1}v^{\alpha}|\nabla v|^{\gamma}v_{ii}^2 &= \sum_{i > 1}v^{\alpha}|\nabla v|^{\gamma}\Big(G_{ii} + Q\Delta v\Big)^2\\
			&= \sum_{i > 1}v^{\alpha}|\nabla v|^{\gamma}G_{ii}^2 + 2Q\sum_{i > 1}v^{\alpha}|\nabla v|^{\gamma}G_{ii}\Delta v + (n - 1)Q^2 v^{\alpha}|\nabla v|^{\gamma}(\Delta v)^2.
		\end{aligned}
	\end{equation*}
	If $(n - 1)Q + S = 1$, then
	\begin{equation*}
		\begin{aligned}
			G_{11} + \sum_{i > 1}G_{ii} = 0,
		\end{aligned}
	\end{equation*}
		and
	\begin{equation*}
		\begin{aligned}
			\sum_{i > 1}v^{\alpha}|\nabla v|^{\gamma}v_{ii}^2 &= \sum_{i > 1}v^{\alpha}|\nabla v|^{\gamma}\Big(G_{ii} + Q\Delta v\Big)^2\\
			&= \sum_{i > 1}v^{\alpha}|\nabla v|^{\gamma}G_{ii}^2 - 2Qv^{\alpha}|\nabla v|^{\gamma}G_{11}\Delta v + (n - 1)Q^2 v^{\alpha}|\nabla v|^{\gamma}(\Delta v)^2.
		\end{aligned}
	\end{equation*}
	It follows that	
	
	\begin{align}\label{8.239.37}
		\Rightarrow & \Big[1 - S^2 + \gamma S - \gamma S^2 - (n - 1)Q^2\Big]v^{\alpha}|\nabla v|^{\gamma}(\Delta v)^2 \notag\\
			&= \left(1 - \frac{1}{n}\right)\Big( v^{\alpha}|\nabla v|^{\gamma}\Delta v v_i\Big)_i - ( v^{\alpha}|\nabla v|^{\gamma}E_{ij}v_j)_i + \frac{\alpha S - \alpha}{1 + \gamma S + 2S} (v^{\alpha - 1}|\nabla v|^{\gamma + 2}v_i)_i \notag\\
			&+ \sum_{i + j > 2} v^{\alpha}|\nabla v|^{\gamma}G_{ij}^2  + \frac{\alpha(\alpha - 1)(1 - S)}{1 + \gamma S + 2S}v^{\alpha - 2}|\nabla v|^{\gamma +4}+  \frac{\alpha(\gamma + 3)}{1 + \gamma S + 2S} v^{\alpha - 1}|\nabla v|^{\gamma + 2} G_{11}\notag\\
			&+ \Big(2\gamma S - \gamma + 2S - 2Q\Big) v^{\alpha} |\nabla v|^{\gamma}G_{11}\Delta v + (1 + \gamma) v^{\alpha}|\nabla v|^{\gamma}G_{11}^2 + \sum_{i > 1} \gamma v^{\alpha}|\nabla v|^{\gamma }G_{i1}^2. 
	\end{align}
		
	\begin{remark}
		The divergence term
		\begin{equation*}
			A := (v^{\alpha}|\nabla v|^{\gamma}v_i \Delta v)_i
		\end{equation*}
		 is important and will be explained later.
	\end{remark}

	\item The term $A$:
	
	\begin{equation*}
		\begin{aligned}
			A &= (v^{\alpha}|\nabla v|^{\gamma}v_i \Delta v)_i\\
			&= \alpha v^{\alpha - 1}|\nabla v|^{\gamma + 2}\Delta v + \gamma v^{\alpha}|\nabla v|^{\gamma - 2} v_i v_j v_{ij}\Delta v + v^{\alpha}|\nabla v|^{\gamma}(\Delta v)^2 + \underset{A_1}{ v^{\alpha}|\nabla v|^{\gamma}v_i (\Delta v)_i}\\
			&= 	\alpha v^{\alpha - 1}|\nabla v|^{\gamma + 2}\Delta v + \gamma v^{\alpha}|\nabla v|^{\gamma } G_{11}\Delta v + (1 + \gamma S) v^{\alpha}|\nabla v|^{\gamma}(\Delta v)^2 + \underset{A_1}{ v^{\alpha}|\nabla v|^{\gamma}v_i (\Delta v)_i}.\\
		\end{aligned}
	\end{equation*}
	Substitute \eqref{lap} into $A_1$,
	\begin{equation*}
		\begin{aligned}
			A_1 &=  v^{\alpha}|\nabla v|^{\gamma}v_i \Big( -Nv^{p} - M|\nabla v|^q\Big)_i\\
			&= -pNv^{\alpha + p - 1}|\nabla v|^{\gamma + 2} - qM v^{\alpha}|\nabla v|^{\gamma + q - 2}v_i v_j v_{ij}.
		\end{aligned}
	\end{equation*}	
	Multiply \eqref{lap}  by $v^{\alpha - 1}|\nabla v|^{\gamma + 2}$
	\begin{equation}
		\begin{aligned}
			-N v^{\alpha + p - 1}|\nabla v|^{\gamma + 2} = v^{\alpha - 1}|\nabla v|^{\gamma + 2}\Delta v  + M v^{\alpha - 1}|\nabla v|^{\gamma + q + 2}.
		\end{aligned}
	\end{equation}
	Therefore we have
	\begin{equation}\label{sub22}
		\begin{aligned}
			A_1 &= - qM v^{\alpha}|\nabla v|^{\gamma + q - 2}v_i v_j v_{ij} + p v^{\alpha - 1}|\nabla v|^{\gamma + 2}\Delta v + p M v^{\alpha - 1}|\nabla v|^{\gamma + q + 2}.\\
		\end{aligned}
	\end{equation}
	On the other hand, we know
	\begin{equation}\label{sub21}
		\begin{aligned}
			 & pMv^{\alpha - 1}|\nabla v|^{\gamma + q + 2} - qM v^{\alpha }|\nabla v|^{\gamma + q - 2}v_i v_j v_{ij}\\
			 &= pM v^{\alpha - 1}|\nabla v|^{\gamma + q + 2} - \frac{ qM}{\gamma + q}v^{\alpha }(|\nabla v|^{\gamma + q})_i v_i\\
			 &= \Big( p + \frac{q\alpha }{\gamma + q} \Big) Mv^{\alpha - 1}|\nabla v|^{\gamma + q + 2} - \frac{q}{\gamma + q}M(v^{\alpha}|\nabla v|^{\gamma + q} v_i)_i\\
			 &+ \frac{q M}{\gamma + q}v^{\alpha}|\nabla v|^{\gamma + q}\Big( -N v^{p} - M |\nabla v|^q \Big)\\
			 &= -\frac{q M}{\gamma + q}(v^{\alpha}|\nabla v|^{\gamma + q} v_i)_i + \Big( p + \frac{q\alpha  }{\gamma + q} \Big) M v^{\alpha - 1}|\nabla v|^{\gamma + q + 2}\\
			 &- \frac{ qMN}{\gamma + q}v^{\alpha + p}|\nabla v|^{\gamma + q} - \frac{ qM^2 }{\gamma + q}v^{\alpha}|\nabla v|^{\gamma + 2q}.
		\end{aligned}
	\end{equation}
	So we get that
	\begin{equation*}
		\begin{aligned}
			A &= (v^{\alpha}|\nabla v|^{\gamma}v_i \Delta v)_i\\
			&= 	\alpha v^{\alpha - 1}|\nabla v|^{\gamma + 2}\Delta v + \gamma v^{\alpha}|\nabla v|^{\gamma } G_{11}\Delta v + (1 + \gamma S) v^{\alpha}|\nabla v|^{\gamma}(\Delta v)^2 + \underset{A_1}{ v^{\alpha}|\nabla v|^{\gamma}v_i (\Delta v)_i}\\
			&= \alpha v^{\alpha - 1}|\nabla v|^{\gamma + 2}\Delta v + \gamma v^{\alpha}|\nabla v|^{\gamma } G_{11}\Delta v + (1 + \gamma S)v^{\alpha}|\nabla v|^{\gamma}(\Delta v)^2\\
			&+ p v^{\alpha - 1}|\nabla v|^{\gamma + 2}\Delta v -\frac{q M}{\gamma + q}(v^{\alpha}|\nabla v|^{\gamma + q} v_i)_i - \frac{ qM^2 }{\gamma + q}v^{\alpha}|\nabla v|^{\gamma + 2q}\\
			&+ \Big( p + \frac{q\alpha  }{\gamma + q} \Big) M v^{\alpha - 1}|\nabla v|^{\gamma + q + 2} - \frac{ q}{\gamma + q}MNv^{\alpha + p}|\nabla v|^{\gamma + q}.
		\end{aligned}
	\end{equation*}
Substitute \eqref{section1equ5} into it:
	\begin{equation*}
		\begin{aligned}
			\Rightarrow& - (1 + \gamma S)v^{\alpha}|\nabla v|^{\gamma}(\Delta v)^2 \\
			&= -A + \gamma v^{\alpha}|\nabla v|^{\gamma } G_{11}\Delta v + (\alpha + p ) \cdot\Bigg[  \frac{1}{1 + \gamma S + 2S} (v^{\alpha - 1}|\nabla v|^{\gamma + 2}v_i)_i - \frac{\alpha - 1}{1 + \gamma S + 2S}v^{\alpha - 2}|\nabla v|^{\gamma +4} \\
			&- \frac{\gamma + 2}{1 + \gamma S + 2S} v^{\alpha - 1}|\nabla v|^{\gamma + 2} G_{11}\Bigg]  -\frac{q M}{\gamma + q}(v^{\alpha}|\nabla v|^{\gamma + q} v_i)_i - \frac{ qM^2 }{\gamma + q}v^{\alpha}|\nabla v|^{\gamma + 2q} \\
			&+ \Big( p + \frac{q\alpha  }{\gamma + q} \Big) M v^{\alpha - 1}|\nabla v|^{\gamma + q + 2}  - \frac{q}{\gamma + q}M Nv^{\alpha + p}|\nabla v|^{\gamma + q},
		\end{aligned}
	\end{equation*}

	\begin{equation}\label{sub3final4}
		\begin{aligned}
			\Rightarrow& - (1 + \gamma S) v^{\alpha}|\nabla v|^{\gamma}(\Delta v)^2\\
			 &= -A +  \frac{\alpha + p}{1 + \gamma S + 2S} (v^{\alpha - 1}|\nabla v|^{\gamma + 2}v_i)_i - \frac{q M}{\gamma + q}(v^{\alpha}|\nabla v|^{\gamma + q} v_i)_i + \gamma v^{\alpha}|\nabla v|^{\gamma } G_{11}\Delta v \\
			 &- ( \alpha + p)\frac{\gamma + 2}{1 + \gamma S + 2S}v^{\alpha - 1}|\nabla v|^{\gamma + 2} G_{11} - (\alpha + p)\frac{\alpha - 1}{1 + \gamma S + 2S}v^{\alpha - 2}|\nabla v|^{\gamma + 4}  - \frac{q }{\gamma + q} M^2 v^{\alpha }|\nabla v|^{\gamma + 2q}\\
			& + \Big( p + \frac{q\alpha }{\gamma + q} \Big)  M v^{\alpha - 1}|\nabla v|^{\gamma + q + 2} - \frac{q}{\gamma + q} M N  v^{\alpha + p}|\nabla v|^{\gamma + q}.
		\end{aligned}
	\end{equation}
	Using the fact that
	\begin{equation}
		\begin{aligned}
			E_{ij}^2 &= v_{ij}^2 - \frac{1}{n}(\Delta v)^2\\
			&= G_{ij}^2 + 2(S - Q)G_{11}\Delta v + \Big[S^2 + (n - 1)Q^2 - \frac{1}{n}\Big](\Delta v)^2,
		\end{aligned}
	\end{equation}	
	and the equation \eqref{sub3final4}, we know that for any $\varepsilon > 0$, it follows that
	
	\begin{align}\label{sub3final5}
		& - (1 + \gamma S - \varepsilon\tau) v^{\alpha}|\nabla v|^{\gamma}(\Delta v)^2 \notag\\
			 &= -A +  \frac{\alpha + p}{1 + \gamma S + 2S} (v^{\alpha - 1}|\nabla v|^{\gamma + 2}v_i)_i - \frac{q M}{\gamma + q}(v^{\alpha}|\nabla v|^{\gamma + q} v_i)_i + \gamma v^{\alpha}|\nabla v|^{\gamma } G_{11}\Delta v \notag\\
			 &- ( \alpha + p)\frac{\gamma + 2}{1 + \gamma S + 2S}v^{\alpha - 1}|\nabla v|^{\gamma + 2} G_{11} - (\alpha + p)\frac{\alpha - 1}{1 + \gamma S + 2S}v^{\alpha - 2}|\nabla v|^{\gamma + 4}  - \frac{q }{\gamma + q} M^2 v^{\alpha }|\nabla v|^{\gamma + 2q} \notag\\
			& + \Big( p + \frac{q\alpha }{\gamma + q} \Big)  M v^{\alpha - 1}|\nabla v|^{\gamma + q + 2} - \frac{q}{\gamma + q} M N  v^{\alpha + p}|\nabla v|^{\gamma + q} \notag\\
			&+ \varepsilon v^{\alpha}|\nabla v|^{\gamma}E_{ij}^2 - \varepsilon v^{\alpha}|\nabla v|^{\gamma}G_{ij}^2 - 2\varepsilon(S - Q)v^{\alpha}|\nabla v|^{\gamma}G_{11}\Delta v ,
	\end{align}
	with $\tau := S^2 + (n - 1)Q^2 - \frac{1}{n}$.
	
	\end{itemize}
	Now we can give the final differential identity. Combining \eqref{8.239.37} and \eqref{sub3final5} to eliminate the term $v^{\alpha}|\nabla v|^{\gamma}(\Delta v)^2$ on the left hand, we obtain
	\begin{align*}
		0& = -A +  \frac{\alpha + p}{1 + \gamma S + 2S} (v^{\alpha - 1}|\nabla v|^{\gamma + 2}v_i)_i - \frac{q M}{\gamma + q}(v^{\alpha}|\nabla v|^{\gamma + q} v_i)_i\\
			&+ \gamma v^{\alpha}|\nabla v|^{\gamma } G_{11}\Delta v - ( \alpha + p)\frac{\gamma + 2}{1 + \gamma S + 2S}v^{\alpha - 1}|\nabla v|^{\gamma + 2} G_{11}\\
			& - (\alpha + p)\frac{\alpha - 1}{1 + \gamma S + 2S}v^{\alpha - 2}|\nabla v|^{\gamma + 4}  - \frac{q }{\gamma + q} M^2 v^{\alpha }|\nabla v|^{\gamma + 2q}\\
			& + \Big( p + \frac{q\alpha }{\gamma + q} \Big)  M v^{\alpha - 1}|\nabla v|^{\gamma + q + 2} - \frac{q}{\gamma + q} M N  v^{\alpha + p}|\nabla v|^{\gamma + q}\\
			&+ \varepsilon v^{\alpha}|\nabla v|^{\gamma}E_{ij}^2 - \varepsilon v^{\alpha}|\nabla v|^{\gamma}G_{ij}^2 - 2\varepsilon(S - Q)v^{\alpha}|\nabla v|^{\gamma}G_{11}\Delta v \\
			 &+ \frac{1 + \gamma S - \varepsilon \tau}{1 - S^2 + \gamma S - \gamma S^2 - (n - 1)Q^2}\Bigg \{ \left(1 - \frac{1}{n}\right)\Big( v^{\alpha}|\nabla v|^{\gamma}\Delta v v_i\Big)_i - ( v^{\alpha}|\nabla v|^{\gamma}E_{ij}v_j)_i \\
			 &+ \frac{\alpha S - \alpha}{1 + \gamma S + 2S} (v^{\alpha - 1}|\nabla v|^{\gamma + 2}v_i)_i\\
			&+ \sum_{i + j > 2} v^{\alpha}|\nabla v|^{\gamma}G_{ij}^2  + \frac{\alpha(\alpha - 1)(1 - S)}{1 + \gamma S + 2S}v^{\alpha - 2}|\nabla v|^{\gamma +4}+  \frac{\alpha(\gamma + 3)}{1 + \gamma S + 2S} v^{\alpha - 1}|\nabla v|^{\gamma + 2} G_{11}\\
			&+ \Big(2\gamma S - \gamma + 2S - 2Q\Big) v^{\alpha} |\nabla v|^{\gamma}G_{11}\Delta v + (1 + \gamma) v^{\gamma}|\nabla v|^{\gamma}G_{11}^2 + \sum_{i > 1} \gamma v^{\alpha}|\nabla v|^{\gamma }G_{i1}^2\Bigg\},
	\end{align*}

	\begin{align}\label{sec1_final}
		\Rightarrow & 0 = W + \varepsilon v^{\alpha}|\nabla v|^{\gamma}E_{ij}^2 \notag\\
			&+ \Bigg[ \gamma  + \frac{1 + \gamma S - \varepsilon \tau}{1 - S^2 + \gamma S - \gamma S^2 - (n - 1)Q^2}\Big(2\gamma S - \gamma + 2S - 2Q\Big) - 2\varepsilon(S - Q) \Bigg]v^{\alpha}|\nabla v|^{\gamma } G_{11}\Delta v \notag\\
			&+ \Bigg[ \frac{1 + \gamma S - \varepsilon \tau}{1 - S^2 + \gamma S - \gamma S^2 - (n - 1)Q^2} - \varepsilon \Bigg]  v^{\alpha}|\nabla v|^{\gamma}\sum_{i + j > 2}^nG_{ij}  ^2 \notag\\
			& + \Bigg[ \frac{1 + \gamma S - \varepsilon \tau}{1 - S^2 + \gamma S - \gamma S^2 - (n - 1)Q^2}(1 + \gamma) - \varepsilon \Bigg]  v^{\alpha}|\nabla v|^{\gamma}\sum_{i = 1}^n G_{1i}^2 \notag\\
			&+ \Bigg[ - \Big( \alpha + p\Big)\frac{\alpha - 1}{1 + \gamma S + 2S} + \frac{1 + \gamma S - \varepsilon \tau}{1 - S^2 + \gamma S - \gamma S^2 - (n - 1)Q^2}\cdot\frac{\alpha(\alpha - 1)(1 - S)}{1 + \gamma S + 2S}   \Bigg]v^{\alpha - 2}|\nabla v|^{\gamma + 4} \notag\\
			&+  \Bigg[ \frac{1 + \gamma S - \varepsilon \tau}{1 - S^2 + \gamma S - \gamma S^2 - (n - 1)Q^2}\cdot  \frac{\alpha(\gamma + 3)}{1 + \gamma S + 2S} - \Big(\alpha + p\Big)\frac{\gamma + 2}{1 + \gamma S + 2S}   \Bigg]v^{\alpha - 1}|\nabla v|^{\gamma + 2} G_{11} \notag\\
			& - \frac{ q }{\gamma + q} M^2 v^{\alpha}|\nabla v|^{\gamma + 2q} + \Big( p + \frac{q\alpha }{\gamma + q}\Big) M v^{\alpha - 1}|\nabla v|^{\gamma + q + 2} - \frac{q}{\gamma + q} M N  v^{\alpha + p}|\nabla v|^{\gamma + q},
	\end{align}
	where $W$ consists of all the divergence terms and $\tau := S^2 + (n - 1)Q^2 - \frac{1}{n}$.
We can rewrite \eqref{sec1_final} as

	\begin{equation}\label{TheFinalEqua1}
		\begin{aligned}
			0 &= W + \varepsilon v^{\alpha}|\nabla v|^{\gamma}E_{ij}^2 + a_1 v^{\alpha}|\nabla v|^{\gamma} \sum_{i + j > 2} G_{ij}^2 + \sum_{i = 1}^n a_2 v^{\alpha}|\nabla v|^{\gamma}G_{1i}^2  +   a_3  v^{\alpha - 2}|\nabla v|^{\gamma + 4} \\
			&+ b_1 v^{\alpha - 1}|\nabla v|^{\gamma + 2}G_{11} + b_2 v^{\alpha}|\nabla v|^{\gamma}G_{11}\Delta v  + c_2 M^2 v^{\alpha}|\nabla v|^{\gamma + 2q} \\
			&- b_3 N v^{\alpha + p - 1}|\nabla v|^{\gamma + 2} - b_4  M v^{\alpha - 1}|\nabla v|^{\gamma + q + 2} + b_5 MN v^{\alpha + p}|\nabla v|^{\gamma + q},\\
		\end{aligned}
	\end{equation}
	with
	\begin{align}
		a_1 &= \frac{1 + \gamma S - \varepsilon \tau}{1 - S^2 + \gamma S - \gamma S^2 - (n - 1)Q^2} - \varepsilon ,\\
			a_2 &=\frac{1 + \gamma S - \varepsilon \tau}{1 - S^2 + \gamma S - \gamma S^2 - (n - 1)Q^2}(1 + \gamma) - \varepsilon ,\\
			a_3 &= - \Big(\alpha + p\Big)\frac{\alpha - 1}{1 + \gamma S + 2S} + \frac{1 + \gamma S - \varepsilon \tau}{1 - S^2 + \gamma S - \gamma S^2 - (n - 1)Q^2}\cdot\frac{\alpha(\alpha - 1)(1 - S)}{1 + \gamma S + 2S} ,\\
			b_1 &= \frac{1 + \gamma S - \varepsilon \tau}{1 - S^2 + \gamma S - \gamma S^2 - (n - 1)Q^2}\cdot  \frac{\alpha(\gamma + 3)}{1 + \gamma S + 2S} - \Big( \alpha + p\Big)\frac{\gamma + 2}{1 + \gamma S + 2S} ,\\
			b_2 &=  \gamma + \frac{1 + \gamma S - \varepsilon \tau}{1 - S^2 + \gamma S - \gamma S^2 - (n - 1)Q^2}\Big(2\gamma S - \gamma + 2S - 2Q\Big) - 2\varepsilon (S - Q) ,\\
			c_2 &= - \frac{q }{\gamma + q},\\
			b_3 &= 0,\\
			b_4 &= - p - \frac{q\alpha }{\gamma + q} ,\\
			b_5 &= - \frac{q}{\gamma + q} = c_2.
	\end{align}

	\section{Conditions}
	
	Multiply (\ref{lap})  by $N v^{\alpha + p}|\nabla v|^{\gamma}$:
	\begin{equation}
		\begin{aligned}
			&N^2 v^{\alpha + 2p}|\nabla v|^\gamma + MN v^{\alpha + p}|\nabla v|^{\gamma + q} \\
			&= -\frac{N}{1 + \gamma S}(v^{\alpha + p}|\nabla v|^{\gamma}v_i)_i + \frac{\alpha + p}{1 + \gamma S}Nv^{\alpha + p - 1}|\nabla v|^{\gamma + 2} + \frac{\gamma}{1 + \gamma S}Nv^{\alpha + p}|\nabla v|^{\gamma}G_{11}.
		\end{aligned}
	\end{equation}
	Multiply (\ref{lap})  by $M v^{\alpha}|\nabla v|^{\gamma + q}$:
	\begin{equation}
		\begin{aligned}
			& M N v^{\alpha + p}|\nabla v|^{\gamma + q} +  M^2 v^{\alpha}|\nabla v|^{\gamma + 2q}\\
	 		&= \frac{- M}{1 + \gamma S + qS} (v^{\alpha}|\nabla v|^{\gamma + q}v_i)_i + \Big( \frac{\alpha}{1 + \gamma S + qS} \Big) M v^{\alpha - 1}|\nabla v|^{\gamma + q + 2}+ \frac{ \gamma + q}{1 + \gamma S + qS} M v^{\alpha}|\nabla v|^{\gamma + q}G_{11} .\\
		\end{aligned}
	\end{equation}
	Multiply (\ref{lap})  by $ v^{\alpha - 1}|\nabla v|^{\gamma + 2}$:
	\begin{equation*}
		\begin{aligned}
			- M v^{\alpha - 1}|\nabla v|^{\gamma + q + 2} &= \frac{1}{1 + \gamma S + 2S} (v^{\alpha - 1}|\nabla v|^{\gamma + 2}v_i)_i - \frac{\alpha - 1}{1 + \gamma S + 2S} v^{\alpha - 2}|\nabla v|^{\gamma +4}\\
			&- \frac{\gamma + 2}{1 + \gamma S + 2S} v^{\alpha - 1}|\nabla v|^{\gamma + 2} G_{11} + N v^{\alpha + p - 1}|\nabla v|^{\gamma + 2}.\\
		\end{aligned}
	\end{equation*}
	Besides, by Cauchy-Schwarz inequality,
	\begin{align*}
		(n - 1)\sum_{i > 1}G_{ii}^2 \geq \left(\sum_{i > 1} G_{ii}\right)^2 = G_{11}^2.
	\end{align*}
		If $\varepsilon > 0$ is small enough, $\gamma \geq 0$, $ b_2 = 0$ and
	\begin{equation*}
		\frac{1 + \gamma S}{1 - S^2 + \gamma S - \gamma S^2 - (n - 1)Q^2} > 0,
	\end{equation*}
	then by \eqref{TheFinalEqua1} we get that
	\begin{align*}
		0 &\geq  W + \varepsilon v^{\alpha}|\nabla v|^{\gamma}E_{ij}^2 + \Big( a_2 + \frac{1}{n - 1}a_1 \Big) v^{\alpha}|\nabla v|^{\gamma}G_{11}^2  + a_3  v^{\alpha - 2}|\nabla v|^{\gamma + 4} + b_1 v^{\alpha - 1}|\nabla v|^{\gamma + 2}G_{11}\\
			& + c_2 M^2 v^{\alpha}|\nabla v|^{\gamma + 2q}- b_3 N v^{\alpha + p - 1}|\nabla v|^{\gamma + 2} - b_4  M v^{\alpha - 1}|\nabla v|^{\gamma + q + 2} + c_2 MN v^{\alpha + p}|\nabla v|^{\gamma + q}\\
			&=  W + \varepsilon v^{\alpha}|\nabla v|^{\gamma}E_{ij}^2 + \Big( a_2 + \frac{1}{n - 1}a_1 \Big) v^{\alpha}|\nabla v|^{\gamma}G_{11}^2  + a_3  v^{\alpha - 2}|\nabla v|^{\gamma + 4} + b_1 v^{\alpha - 1}|\nabla v|^{\gamma + 2}G_{11}\\
			& + c_2 M^2 v^{\alpha}|\nabla v|^{\gamma + 2q} - b_3 N v^{\alpha + p - 1}|\nabla v|^{\gamma + 2} - b_4  M v^{\alpha - 1}|\nabla v|^{\gamma + q + 2} + c_2 MN v^{\alpha + p}|\nabla v|^{\gamma + q}\\
			&+ T \Bigg( N^2 v^{\alpha + 2p}|\nabla v|^\gamma + MN v^{\alpha + p}|\nabla v|^{\gamma + q} - \frac{\alpha + p}{1 + \gamma S}Nv^{\alpha + p - 1}|\nabla v|^{\gamma + 2} \\
			&- \frac{\gamma}{1 + \gamma S}Nv^{\alpha + p}|\nabla v|^{\gamma}G_{11} \Bigg)\\
			&+ U \Bigg( M N v^{\alpha + p}|\nabla v|^{\gamma + q} +  M^2 v^{\alpha}|\nabla v|^{\gamma + 2q}\\
	 		&-  \frac{\alpha}{1 + \gamma S + qS} M v^{\alpha - 1}|\nabla v|^{\gamma + q + 2} - \frac{ \gamma + q}{1 + \gamma S + qS} M v^{\alpha }|\nabla v|^{\gamma + q}G_{11} \Bigg)\\
	 		&+ P \Bigg( - N v^{\alpha + p - 1}|\nabla v|^{\gamma + 2} + \frac{\alpha - 1}{1 + \gamma S + 2S} v^{\alpha - 2}|\nabla v|^{\gamma +4}\\
			&+ \frac{\gamma + 2}{1 + \gamma S + 2S} v^{\alpha - 1}|\nabla v|^{\gamma + 2} G_{11} - M v^{\alpha - 1}|\nabla v|^{\gamma + q + 2}\Bigg),
	\end{align*}
	
	\begin{align*}
		\Rightarrow 0 &\geq  W + \varepsilon v^{\alpha}|\nabla v|^{\gamma}E_{ij}^2 + \Big( a_2 + \frac{1}{n - 1}a_1 \Big) v^{\alpha}|\nabla v|^{\gamma}G_{11}^2  \\
			&+ \Bigg( a_3 + P \frac{\alpha - 1}{1 + \gamma S + 2S}\Bigg)  v^{\alpha - 2}|\nabla v|^{\gamma + 4} + \Big( b_1 + P  \frac{\gamma + 2}{1 + \gamma S + 2S}\Big) v^{\alpha - 1}|\nabla v|^{\gamma + 2}G_{11}\\
			&+ T N^2 v^{\alpha + 2p}|\nabla v|^\gamma  + \Big( c_2 + U\Big) M^2 v^{\alpha}|\nabla v|^{\gamma + 2q} - \Bigg[ b_3  + P + \frac{T(\alpha + p)}{1 + \gamma S}\Bigg]N v^{\alpha + p - 1}|\nabla v|^{\gamma + 2} \\
			&- \frac{T\gamma}{1 + \gamma S}Nv^{\alpha + p}|\nabla v|^{\gamma}G_{11} - U\frac{ \gamma + q}{1 + \gamma S + qS} M v^{\alpha}|\nabla v|^{\gamma + q}G_{11} \\
			&- \Bigg( b_4 + P + U \frac{\alpha }{1 + \gamma S + qS} \Bigg) M v^{\alpha - 1}|\nabla v|^{\gamma + q + 2} + \Big( c_2 + T + U\Big) MN v^{\alpha + p}|\nabla v|^{\gamma + q}.\\
	\end{align*}
	Define $B_0:=  a_2 + \frac{1}{n - 1}a_1  $ and if $B_0 > 0$ , then
	\begin{align*}
		\Rightarrow & 0 \geq W + \varepsilon v^{\alpha}|\nabla v|^{\gamma}E_{ij}^2\\
			& - \frac{1}{4B_0}v^{\alpha}|\nabla v|^{\gamma}\Bigg[ \Big( b_1 + P  \frac{\gamma + 2}{1 + \gamma S + 2S}\Big) v^{- 1}|\nabla v|^{ 2} - \frac{T\gamma}{1 + \gamma S}Nv^p - U\frac{ \gamma + q}{1 + \gamma S + qS} M |\nabla v|^{q} \Bigg]^2\\
			&+ \Bigg( a_3 + P \frac{\alpha - 1}{1 + \gamma S + 2S} \Bigg)  v^{\alpha - 2}|\nabla v|^{\gamma + 4} + T N^2 v^{\alpha + 2p}|\nabla v|^\gamma  + \Big( c_2 + U\Big) M^2 v^{\alpha}|\nabla v|^{\gamma + 2q} \\
			&- \Bigg[ b_3  + P + \frac{T(\alpha + p)}{1 + \gamma S}\Bigg]N v^{\alpha + p - 1}|\nabla v|^{\gamma + 2}- \Bigg( b_4 + P + U  \frac{\alpha}{1 + \gamma S + qS} \Bigg) M v^{\alpha - 1}|\nabla v|^{\gamma + q + 2}\\
			&+ \Big( c_2 + T + U\Big) MN v^{\alpha + p}|\nabla v|^{\gamma + q}.\\
	\end{align*}

	\begin{align}\label{0824925}
		\Rightarrow 0 &\geq W + \varepsilon v^{\alpha}|\nabla v|^{\gamma}E_{ij}^2 + \Bigg[ a_3 + P \frac{\alpha - 1}{1 + \gamma S + 2S}  - \frac{1}{4B_0} \Big( b_1 + P  \frac{\gamma + 2}{1 + \gamma S + 2S}\Big)^2 \Bigg]  v^{\alpha - 2}|\nabla v|^{\gamma + 4}\notag \\
			&+ \Bigg[ T - \frac{1}{4B_0}\left(\frac{T\gamma}{1 + \gamma S}\right)^2\Bigg] N^2 v^{\alpha + 2p}|\nabla v|^{\gamma} \notag\\
			&+ \Bigg[ c_2 + U - \frac{1}{4B_0}\Big(U\frac{ \gamma + q}{1 + \gamma S + qS}\Big)^2 \Bigg] M^2 v^{\alpha}|\nabla v|^{\gamma + 2q}\notag \\
			&- \Bigg[ \frac{T(\alpha + p)}{1 + \gamma S} + P - \frac{1}{2B_0}\frac{T\gamma}{1 + \gamma S}\Big( b_1 + P\frac{\gamma + 2}{1 + \gamma S + 2S}\Big) \Bigg] N v^{\alpha + p - 1}|\nabla v|^{\gamma + 2}\notag\\
			&- \Bigg[ b_4 + P + U  \frac{\alpha}{1 + \gamma S + qS}  - \frac{1}{2B_0} \Big( b_1 + P  \frac{\gamma + 2}{1 + \gamma S + 2S}\Big) U\frac{ \gamma + q}{1 + \gamma S + qS} \Bigg] M v^{\alpha - 1}|\nabla v|^{\gamma + q + 2}\notag\\
			&+ \Big( c_2 + T + U - \frac{1}{2B_0}\frac{T\gamma}{1 + \gamma S}\cdot U\frac{\gamma + q}{1 + \gamma S + qS} \Big)MN v^{\alpha + p}|\nabla v|^{\gamma + q}.
	\end{align}
	Define
	\begin{numcases}{}
			\mathbb K_1 := a_3 + P\Big(\frac{\alpha - 1}{1 + \gamma S + 2S} \Big) - \frac{1}{4B_0} \Big( b_1 + P  \frac{\gamma + 2}{1 + \gamma S + 2S}\Big)^2  ,\\
			\mathbb K_2 := T - \frac{1}{4B_0}\left(\frac{T\gamma}{1 + \gamma S}\right)^2 ,\\
			\mathbb K_3 := c_2 + U - \frac{1}{4B_0}\Big(U\frac{ \gamma + q}{1 + \gamma S + qS}\Big)^2 ,\\
			\mathbb K_4 :=  -\frac{T(\alpha + p)}{1 + \gamma S} - P + \frac{1}{2B_0}\frac{T\gamma}{1 + \gamma S}\Big( b_1 + P\frac{\gamma + 2}{1 + \gamma S + 2S}\Big) ,\\
			\mathbb K_5 :=  -b_4 - P - U\Big( \frac{\alpha }{1 + \gamma S + qS} \Big) + \frac{1}{2B_0} \Big( b_1 + P  \frac{\gamma + 2}{1 + \gamma S + 2S}\Big) U\frac{\gamma + q}{1 + \gamma S + qS}  ,\quad \quad\\
			\mathbb K_6 :=   c_2 + T + U - \frac{1}{2B_0}\frac{T\gamma}{1 + \gamma S}\cdot U\frac{\gamma + q}{1 + \gamma S + qS}  .
		\end{numcases}
	Without loss of generality, we can let $N = 1$.	 Then \eqref{0824925} becomes
	\begin{equation}\label{0824927}
		\begin{aligned}
			\Rightarrow 0 &\geq W + \varepsilon v^{\alpha}|\nabla v|^{\gamma}E_{ij}^2 + \mathbb K_1  v^{\alpha - 2}|\nabla v|^{\gamma + 4} + \mathbb K_2 v^{\alpha + 2p}|\nabla v|^{\gamma} + \mathbb K_3 M^2 v^{\alpha}|\nabla v|^{\gamma + 2q} \\
			&+ \mathbb K_4 v^{\alpha + p - 1}|\nabla v|^{\gamma + 2}+ \mathbb K_5 M v^{\alpha - 1}|\nabla v|^{\gamma + q + 2} + \mathbb K_6 M v^{\alpha + p}|\nabla v|^{\gamma + q}.\\
		\end{aligned}
	\end{equation}
	Note that by Cauchy-Schwarz inequality, we have $\mathbb K_6 \geq \mathbb K_2 + \mathbb K_3$.
	In general, we hope that the following condition holds for $\varepsilon_1 > 0$ small enough:
		
		\begin{equation}\label{condition08240930}
			\begin{aligned}
				&(\mathbb K_1 - \varepsilon_1)  v^{\alpha - 2}|\nabla v|^{\gamma + 4} + (\mathbb K_2 - \varepsilon_1) v^{\alpha + 2p}|\nabla v|^{\gamma} + (\mathbb K_3 - \varepsilon_1) M^2 v^{\alpha}|\nabla v|^{\gamma + 2q} \\
			&+ \mathbb K_4 v^{\alpha + p - 1}|\nabla v|^{\gamma + 2}+ \mathbb K_5 M v^{\alpha - 1}|\nabla v|^{\gamma + q + 2} + \mathbb K_6 M v^{\alpha + p}|\nabla v|^{\gamma + q} \geq 0.
			\end{aligned}
		\end{equation}
		If the condition is satisfied, then for any $\varepsilon > 0$ small enough, we have
	\begin{equation}\label{sub4_eq1}
		\begin{aligned}
			&\varepsilon v^{\alpha}|\nabla v|^{\gamma}|E_{ij}|^2 + \varepsilon v^{\alpha - 2}|\nabla v|^{\gamma + 4}  + \varepsilon v^{\alpha + 2p}|\nabla v|^{\gamma } + \varepsilon v^{\alpha}|\nabla v|^{\gamma + 2q}\\
			& \leq B_1(v^{\alpha} |\nabla v|^{\gamma}v_j E_{ij})_i + B_2 (v^{\alpha + p}|\nabla v|^{\gamma}v_i)_i + B_3 (v^{\alpha - 1}|\nabla v|^{\gamma + 2}v_i)_i + B_4(v^{\alpha}|\nabla v|^{\gamma + q} v_i)_i.
		\end{aligned}
	\end{equation}
			
		\section{Young inequality and cut-off functions}
	In this part, we assume \eqref{sub4_eq1} holds and prove that $|\nabla v| \equiv 0$. Define $\eta$ is a smooth cut-off function, satisfying that
	\begin{equation*}
		\begin{aligned}
			&\eta \equiv 1 \quad \text{in} \quad B_{R},\\
			&\eta \equiv 0 \quad \text{in} \quad \mathbb R^n\backslash B_{2R}.\\
		\end{aligned}
	\end{equation*}
	 Multiply \eqref{sub4_eq1} by $\eta^{\delta}$ and integrate over $\mathbb R^n$,
	\begin{equation*}
		\begin{aligned}
			\text{RHS} &=  B_1\delta\int v^{\alpha} |\nabla v|^{\gamma}v_j E_{ij}\eta^{\delta - 1}\eta_i +  B_2\delta \int v^{\alpha + p}|\nabla v|^{\gamma}v_i\eta^{\delta - 1}\eta_i \\
			&+ B_3\delta\int v^{\alpha - 1}|\nabla v|^{\gamma + 2}v_i \eta^{\delta - 1}\eta_i + B_4 \delta \int v^{\alpha}|\nabla v|^{\gamma + q} v_i\eta^{\delta - 1}\eta_i\\
			&\leq \frac{\varepsilon}{2}\int v^{\alpha }|\nabla v|^{\gamma}|E_{ij}|^2\eta^{\delta} + C \int v^{\alpha}|\nabla v|^{\gamma + 2}\eta^{\delta - 2}|\nabla \eta|^2\\
			&+ \frac{\varepsilon}{2}\int v^{\alpha + 2p}|\nabla v|^{\gamma}\eta^{\delta} + C\int v^{\alpha}|\nabla v|^{\gamma + 2}\eta^{\delta - 2}|\nabla \eta|^2\\
			&+ \frac{\varepsilon}{2}\int v^{\alpha - 2}|\nabla v|^{\gamma + 4}\eta^{\delta} + C\int v^{\alpha}|\nabla v|^{\gamma + 2}\eta^{\delta - 2}|\nabla \eta|^2\\
			&+ \frac{\varepsilon}{2} \int v^{\alpha }|\nabla v|^{\gamma + 2q}\eta^{\delta} + C\int v^{\alpha}|\nabla v|^{\gamma + 2}\eta^{\delta - 2}|\nabla \eta|^2.\\
		\end{aligned}
	\end{equation*}
	Therefore we get that
	
	\begin{equation}
		\begin{aligned}
			\Rightarrow &\int v^{\alpha + 2p} |\nabla v|^{\gamma}\eta^{\delta} + \int v^{\alpha - 2}|\nabla v|^{\gamma + 4}\eta^{\delta}\\
		 &\leq C\int v^{\alpha}|\nabla v|^{\gamma + 2}\eta^{\delta - 2}|\nabla \eta|^2.
		\end{aligned}
	\end{equation}
	Define $p_1, q_1, \sigma_1 \geq 0$, such that
	\begin{equation}\label{parameter3}
		\frac{1}{p_1} + \frac{1}{q_1} + \frac{1}{\sigma_1} = 1, \text{ and } p_1, q_1, \sigma_1 \geq 0.
	\end{equation}
	So by Young inequality, we know
	\begin{equation*}
		\begin{aligned}
			&\int v^{\alpha}|\nabla v|^{\gamma + 2}\eta^{\delta - 2}|\nabla \eta|^2\\
			&= \int v^{\alpha - A}|\nabla v|^{\gamma + 2 - B} v^{A}|\nabla v|^{ B} \eta^{\delta - 2}|\nabla \eta|^2\\
			&\leq \varepsilon_0 \int v^{\alpha - 2}|\nabla v|^{\gamma + 4}\eta^\delta + \varepsilon_0 \int v^{\alpha + 2p}|\nabla v|^{\gamma}\eta^{\delta}  + C\int \eta^{\delta - 2\sigma_1}|\nabla \eta|^{2\sigma_1}\\
			\Rightarrow & \int v^{\alpha +  2p}|\nabla v|^{\gamma}\eta^{\delta} + \int v^{\alpha - 2}|\nabla v|^{\gamma + 4}\eta^\delta \\
			&\leq C\int \eta^{\delta - 2\sigma_1}|\nabla \eta|^{2\sigma_1}\\
			&\leq CR^{n - 2\sigma_1} \rightarrow 0, \text{as R tends to infinity}.
		\end{aligned}
	\end{equation*}
	The last two steps use the condition that
	\begin{equation}\label{parameter5}
		\begin{aligned}
			\delta - 2\sigma_1 >0,\\
			n - 2\sigma_1  < 0.
		\end{aligned}
	\end{equation}
	So we always take $\delta$ large enough.
	In conclusion, we need
	\begin{numcases}{}
		\label{0824young1}\frac{\alpha + 2p}{A} = \frac{\gamma}{B} = p_1,\\
		\label{0824young2}\frac{\alpha - 2}{\alpha - A} = \frac{\gamma + 4}{\gamma + 2 - B} = q_1,\\
		1 - \frac{2}{n} < \frac{1}{p_1} + \frac{1}{q_1} < 1.
	\end{numcases}

	Therefore, as long as the condition \eqref{condition08240930} is satisfied, and there exists such $p_1, q_1, \sigma_1$, then $v \equiv 0$.

	\section{Proof of Theorem \ref{Theorem1}}
	In this section, we choose $\gamma = 0, S = Q = \frac{1}{n}$, then
	\begin{equation*}
		\begin{aligned}
			G_{ij}&= E_{ij},\\
			\tau &= 0,\\
			a_1 &= \frac{n}{n - 1} -\varepsilon ,\\
			a_3 &= -\frac{n}{n + 2}p(\alpha - 1) ,\\
			b_1 &= \frac{n}{n - 1}\alpha - \frac{2np}{n + 2} ,\\
			b_2 &= 0,\\
			c_2 &= - 1,\\
			b_3 &= 0,\\
			b_4 &= - p - \alpha ,\\
			b_5 &= - 1,\\
			B_0 &= \frac{n}{n - 1}\left(\frac{n}{n - 1} - \varepsilon\right).
		\end{aligned}
	\end{equation*}
	And we get
	
	\begin{numcases}{}
			\mathbb K_1 = a_3 + P\cdot\frac{\alpha - 1}{1  + 2S}  - \frac{1}{4B_0} \Big( b_1 +   \frac{ 2P}{1  + 2S}\Big)^2 , \\
			\mathbb K_2 = T ,\\
			\mathbb K_3 = c_2 + U - \frac{1}{4B_0}\Big(U\frac{ q}{1 + qS}\Big)^2 , \\
			\mathbb K_4 = -T(\alpha + p) - P ,\\
			 \mathbb K_5 = - b_4 - P - U\Big( \frac{\alpha }{1 + qS} \Big) + \frac{1}{2B_0} \Big( b_1 + P  \frac{2}{1 + 2S}\Big) U\frac{q}{1 + qS}  ,\\
			\mathbb K_6 =  c_2 + T + U .
	\end{numcases}	
	By direct computation, we deduce that if $n\geq 3$, $\varepsilon > 0$ is small enough and
	\begin{numcases}{}
			\alpha =  -\frac{2(p - P)}{n + 2},\\
			0 < p - P < \frac{n + 2}{n - 2},	
	\end{numcases}
	then $\mathbb K_1 > 2\varepsilon$. Besides we know that if $P = \varepsilon_0, T = \frac{9\varepsilon_0}{\alpha + p}$ where $\varepsilon_0 > 0$ is small , then
	
	\begin{align}
		&\alpha + p = \frac{np}{n + 2} + \frac{2P}{n + 2} \in (0, 4),\\
		&\mathbb K_2 = T > 0.
	\end{align}
	 At this time, we can solve the equations \eqref{0824young1} and \eqref{0824young2} :

	\begin{numcases}{}
		\frac{\alpha + 2p}{A} = p_1,\\
			\frac{\alpha - 2}{\alpha - A} = 2 = q_1.
	\end{numcases}
	We obtain that
	\begin{equation*}
		\begin{aligned}
			A &= \frac{\alpha + 2}{2},
		\end{aligned}
	\end{equation*}
	and
	\begin{equation*}
		\begin{aligned}
			\frac{1}{p_1} + \frac{1}{q_1} &= \frac{1}{2} + \frac{1}{2}\cdot\frac{\alpha + 2}{\alpha + 2p}.
		\end{aligned}
	\end{equation*}
	If $p > 1$ and $p - P < \frac{n + 2}{n - 2}$, then $\alpha + 2p > \alpha + 2 \geq 0$. So we have
	\begin{equation*}
		\frac{1}{p_1} + \frac{1}{q_1} < 1,
	\end{equation*}
	and
	\begin{align*}
		p_1 , q_1 > 0.
	\end{align*}
	Besides since $\alpha + 2p > 0$, then
	\begin{equation*}
		\begin{aligned}
			\frac{1}{p_1} + \frac{1}{q_1} > 1 - \frac{2}{n}\Leftrightarrow &\frac{\alpha + 2}{2(\alpha + 2p)} > \frac{1}{2} - \frac{2}{n},\\
			\Leftrightarrow & -\frac{p - P}{n + 2} + 1 > \left[-\frac{2(p - P)}{n + 2} + 2p\right]\left(\frac{1}{2} - \frac{2}{n}\right),\\
			\Leftrightarrow &1 > p\cdot \frac{n^2 - 2n - 4}{n(n + 2)} - \frac{4P}{n(n + 2)} .
		\end{aligned}
	\end{equation*}
	In fact, we have
	\begin{equation*}
		\begin{aligned}
			p\cdot \frac{n^2 - 2n - 4}{n(n + 2)} &\leq \frac{n + 2}{n - 2} \cdot \frac{n^2 - 2n - 4}{n(n + 2)}\\
			&= \frac{n^2 - 2n - 4}{n^2 - 2n} < 1.
		\end{aligned}
	\end{equation*}
	Thus when $|P|$ is small enough, we have
	\begin{equation*}
		\frac{1}{p_1} + \frac{1}{q_1} > 1 - \frac{2}{n}.
	\end{equation*}
		Now we can prove Theorem \ref{Theorem1}	:
	\begin{proof}[Proof of Theorem \ref{Theorem1}	]
		In this case, we choose
		 $\gamma = 0, S = \frac{1}{n}$. Since $1 < p \leq \frac{n + 2}{n - 2}$, we can choose $\alpha =  -\frac{2(p - P)}{n + 2}$, $P = \varepsilon_0$ and $T = \frac{9\varepsilon_0}{\alpha + p} > 2\varepsilon_0$ , then there exists $\varepsilon > 0$(depending on $\varepsilon_0$) small such that
		$\mathbb K_1 > 2\varepsilon$ and $\mathbb K_2 = T > 2\varepsilon_0$.
	In order to get the condition \eqref{condition08240930}, we only need:	
	\begin{equation}\label{condition08241010}
			\begin{aligned}
				&   \varepsilon_0 v^{\alpha + 2p}|\nabla v|^{\gamma} + (\mathbb K_3 - \varepsilon_0) M^2 v^{\alpha}|\nabla v|^{\gamma + 2q} \\
			&- 10\varepsilon_0 v^{\alpha + p - 1}|\nabla v|^{\gamma + 2} + \mathbb K_5 M v^{\alpha - 1}|\nabla v|^{\gamma + q + 2} + \mathbb K_6 M v^{\alpha + p}|\nabla v|^{\gamma + q} \geq 0.
			\end{aligned}
		\end{equation}
		Firstly we use the terms $ v^{\alpha + p}|\nabla v|^{\gamma + q}$ and $v^{\alpha  - 1}|\nabla v|^{\gamma + q + 2}$ to control the term $v^{\alpha + p - 1}|\nabla v|^{\gamma + 2}$.
	Let $\varepsilon_0 > 0$ be much smaller than $M$ such that
	\begin{align*}
		\frac{ 5\varepsilon_0 q}{M^{\frac{2}{q}}}\Big(\frac{2}{2 - q}\Big)^{-\frac{2 - q}{q}} < \varepsilon_0^{\frac{1}{2}},
	\end{align*}
	 then using the following Young inequality:
	\begin{equation*}
		\begin{aligned}
			v^{\alpha + p - 1}|\nabla v|^{\gamma + 2}\leq \frac{q}{2}\Big(M\cdot\frac{2}{2 - q}\Big)^{-\frac{2 - q}{q}}  v^{\alpha + p}|\nabla v|^{\gamma + q} +  M v^{\alpha  - 1}|\nabla v|^{\gamma + q + 2},
		\end{aligned}
	\end{equation*}
	with $(\frac{2}{q}, \frac{2}{2 - q})$, we get that
	\begin{align*}
		10\varepsilon_0 v^{\alpha + p - 1}|\nabla v|^{\gamma + 2} &\leq \frac{ 5\varepsilon_0 q}{M^{\frac{2}{q}}}\Big(\frac{2}{2 - q}\Big)^{-\frac{2 - q}{q}}  Mv^{\alpha + p}|\nabla v|^{\gamma + q} + 10\varepsilon_0 M v^{\alpha  - 1}|\nabla v|^{\gamma + q + 2}\\
		&\leq \varepsilon_0^{\frac{1}{2}}  Mv^{\alpha + p}|\nabla v|^{\gamma + q} + 10\varepsilon_0 M v^{\alpha  - 1}|\nabla v|^{\gamma + q + 2}.
	\end{align*}
	Then \eqref{condition08241010} becomes
	\begin{equation}\label{condition08241036}
			\begin{aligned}
				&\varepsilon_0 v^{\alpha + 2p}|\nabla v|^{\gamma} + (\mathbb K_3 - \varepsilon_0) M^2 v^{\alpha}|\nabla v|^{\gamma + 2q} \\
			& + (\mathbb K_5 - 10\varepsilon_0) M v^{\alpha - 1}|\nabla v|^{\gamma + q + 2} + (\mathbb K_6 - \varepsilon_0^{\frac{1}{2}} )  M v^{\alpha + p}|\nabla v|^{\gamma + q} \geq 0.
			\end{aligned}
		\end{equation}

	Next let us consider $\mathbb K_5$:
	\begin{equation*}
		\begin{aligned}
			\mathbb K_5 &=  \frac{n}{n + 2}p - U  \frac{\alpha }{1  + qS} + \frac{1}{2B_0} b_1 U\frac{q}{1 + qS}\\
			&= \Big[\frac{2p}{n + 2} - \frac{(n - 1)^2}{2n^2}q\cdot\frac{2n^2p}{(n - 1)(n + 2)} \Big]\frac{U}{1 + qS} + \frac{n}{n + 2}p + O(\varepsilon_0)\\
			&= \Big[  2 - (n - 1)q \Big]\frac{p}{n + 2}\cdot\frac{U}{1 + qS} +  \frac{n}{n + 2}p + O(\varepsilon_0) .\\
		\end{aligned}
	\end{equation*}
	Note that there always exists $U > 0$ such that $\mathbb K_3 > \varepsilon_0$. For $n \geq 3, q > 1$, we know
	\begin{equation*}
		 2 - (n - 1)q < 0.
	\end{equation*}
	For the easier case, we hope that there exists $\varepsilon_0 > 0$ small enough such that the following hold
	\begin{numcases}{}
		  \mathbb K_3 - \varepsilon_0 \geq 0,\\
			\mathbb K_5 - \varepsilon^{\frac{1}{2}}_0  \geq 0.
	\end{numcases}
	If so, then the condition \eqref{condition08241036} also holds for any $M > 0$. We find that if $3 \leq n < 5$, we can choose $U = 1 + \frac{q}{n}$; if $n = 5, 6$, we can choose $U = 1 + \frac{2q}{n}$. But for $n \geq 7$, the above conditions will not hold at the same time. As a result, we need to consider the case for $n \geq 7$
	
	\begin{numcases}{}
		  \mathbb K_3 \leq 0,\\
			\mathbb K_5 - \varepsilon^{\frac{1}{2}}_0   \geq 0.
	\end{numcases}
	We will use $v^{\alpha - 1}|\nabla v|^{\gamma + q + 2} $ and $v^{\alpha + p}|\nabla v|^{\gamma + q}$ to control the term $v^{\alpha}|\nabla v|^{\gamma + 2q}$.
	Note that by Young inequality, we have
		\begin{equation*}
		\begin{aligned}
			M^2v^{\alpha}|\nabla v|^{\gamma + 2q} \leq K M v^{\alpha - 1}|\nabla v|^{\gamma + q + 2} + \frac{\Big(K\cdot \frac{p + 1}{p}\Big)^{-p}}{p + 1} M^{p + 2} v^{\alpha + p}|\nabla v|^{\gamma + q},
		\end{aligned}
	\end{equation*}
	with $(\frac{p + 1}{p}, p + 1)$. We only need that
	\begin{numcases}{}
		\mathbb K_5 - 10\varepsilon_0 + K \mathbb K_3  = 0,\\
		\mathbb K_6 - \varepsilon^{\frac{1}{2}}_0 + \frac{\Big(K\cdot \frac{p + 1}{p}\Big)^{-p}}{p + 1} M^{p + 1} \mathbb K_3  > 0.
	\end{numcases}
	
	\begin{equation}\label{thm1_final}
		\begin{aligned}
			&\Rightarrow \mathbb K_6 - \varepsilon^{\frac{1}{2}}_0  - \frac{p^{p}}{(p + 1)^{p + 1}} M^{p + 1}(- \mathbb K_3)^{p + 1}(\mathbb K_5 - 10\varepsilon_0)^{-p} > 0.
		\end{aligned}
	\end{equation}
	We still choose $U = 1 + \frac{q}{n}$, then
	\begin{align*}
		\mathbb K_6 = \frac{q}{n} + \frac{10\varepsilon_0}{\alpha + p} > 0,
	\end{align*}
	so the condition \eqref{thm1_final} becomes:
	\begin{equation*}
		\begin{aligned}
			& \frac{q}{n} > \frac{p^{p}}{(p + 1)^{p + 1}} M^{p + 1}\Bigg[ \frac{(n - 1)^2q^2}{4n^2} - \frac{q}{n} \Bigg]^{p + 1} \Bigg[ p - \frac{(n - 1)pq}{n + 2} \Bigg]^{-p},\\
			\Leftrightarrow &M < \left(\frac{q}{n}\right)^{\frac{1}{p + 1}} \frac{p + 1}{p^{\frac{p}{p + 1}}}\Bigg[ \frac{(n - 1)^2q^2}{4n^2} - \frac{q}{n} \Bigg]^{- 1} \Bigg[ p - \frac{(n - 1)pq}{n + 2} \Bigg]^{\frac{p}{p + 1}}\\
			&=  (p + 1)\Bigg[ \frac{(n - 1)^2q}{4n} - 1 \Bigg]^{- 1} \Bigg[ \frac{n}{q} - \frac{n(n - 1)}{n + 2} \Bigg]^{\frac{q}{2}}\\
			&=:M_1.
		\end{aligned}
	\end{equation*}

	\end{proof}

	\section{Proof of Theorem \ref{Theorem2} }\label{Proof of Thm3}
	
	In this section we consider the case that $n \geq 7$ and $M$ is large enough. At this case, we can not directly choose $\gamma = 0$ any more.
	 Recall that in \eqref{0824young1} and \eqref{0824young2} :

	\begin{equation*}
		\begin{aligned}
			B &= \gamma \cdot \frac{\alpha + \gamma + 2}{(\gamma + 4)p + 2\alpha + \gamma},
		\end{aligned}
	\end{equation*}
	\begin{equation*}
		\begin{aligned}
			G(p) := \frac{1}{p_1} + \frac{1}{q_1} &= \frac{(\gamma + 2)p + 2\alpha + \gamma + 2}{(\gamma + 4)p + 2\alpha + \gamma}.
		\end{aligned}
	\end{equation*}
	We hope that
	\begin{equation*}
		\alpha + \gamma + 2 > 0,
	\end{equation*}
	then if $\gamma > 0$ we have
	\begin{align*}
		&(\gamma + 4)p + 2\alpha + \gamma > \gamma + 4 + 2\alpha + \gamma > 0,\\
		&B > 0\Rightarrow p_1 > 0,\\
		&B < \frac{\gamma}{2} < \gamma + 2 \Rightarrow q_1 > 0.
	\end{align*}
	Besides $G(p)$ is decreasing and
	\begin{equation*}
		G(p) < G(1) = 1.
	\end{equation*}
	Thus if we choose
	\begin{equation*}
		\begin{aligned}
			\gamma = n - 4,\\
			\alpha + \gamma + 2 > 0,
		\end{aligned}
	\end{equation*}
	then
	\begin{equation*}
		1 - \frac{2}{n} < G(p) < 1.
	\end{equation*}
	So in this section, we only need to show the condition \eqref{condition08240930} holds. When $\gamma \neq 0$, solving the equation
	\begin{equation*}
		b_2 = 0,
	\end{equation*}
	we get that
	\begin{align*}
		b_2 &=  \gamma + \frac{1 + \gamma S - \varepsilon \tau}{1 - S^2 + \gamma S - \gamma S^2 - (n - 1)Q^2}\Big(2\gamma S - \gamma + 2S - 2Q\Big) - 2\varepsilon (S - Q) ,\\
	\end{align*}
	\begin{equation*}
		\begin{aligned}
			\Rightarrow 0 &= \gamma - \gamma S^2 + \gamma^2S - \gamma^2S^2 - \frac{\gamma}{n - 1}(1 - S)^2 + (1 + \gamma S - \varepsilon \tau)\Big( 2\gamma S - \gamma + \frac{2nS}{n - 1} - \frac{2}{n - 1} \Big)\\
			&- 2\varepsilon\cdot \frac{nS - 1}{n - 1}\Big[ 1 - S^2 + \gamma S - \gamma S^2 - \frac{(1 - S)^2}{n - 1}\Big]\\
			&=  2\varepsilon\cdot \frac{n}{n - 1}\Big(    \gamma  + \frac{n}{n - 1}\Big)S^3 + ...
		\end{aligned}
	\end{equation*}
	Since $2\varepsilon\cdot \frac{n}{n - 1}\Big(    \gamma  + \frac{n}{n - 1}\Big) \neq 0$, the cubic equation on the right hand with respect to $S$ at least  has one solution $S = S(n, \gamma, \varepsilon)$. Next, we need to estimate the solution when $\varepsilon > 0$ is small enough. If $\varepsilon = 0$, then
	\begin{equation*}
		\begin{aligned}
			0 &= \gamma - \gamma S^2 + \gamma^2S - \gamma^2S^2 - \frac{\gamma}{n - 1}(1 - S)^2 + (1 + \gamma S )\Big( 2\gamma S - \gamma + \frac{2nS}{n - 1} - \frac{2}{n - 1} \Big)\\
			&= \Big( -\gamma - \gamma^2 - \frac{\gamma}{n - 1} + 2\gamma^2 + \frac{2n\gamma}{n - 1}\Big)S^2\\
			&+ \Big( \gamma^2 + \frac{2\gamma}{n - 1} - \gamma^2 - \frac{2\gamma}{n - 1} + 2\gamma + \frac{2n}{n - 1} \Big)S\\
			&+ \gamma - \frac{\gamma}{n - 1} - \gamma - \frac{2}{n - 1}\\
			&= \Big(  \gamma^2 + \frac{n\gamma}{n - 1}  \Big)S^2 + 2\Big(\gamma + \frac{n}{n - 1} \Big)S - \frac{\gamma + 2}{n - 1}.  \\
		\end{aligned}
	\end{equation*}
	Choose
	\begin{equation*}
		\begin{aligned}
			S &= \frac{-\gamma - \frac{n}{n - 1} + \sqrt{\Big(\gamma + \frac{n}{n - 1} \Big)^2 + \Big(  \gamma^2 + \frac{n\gamma}{n - 1}  \Big)\frac{\gamma + 2}{n - 1}}}{ \gamma^2 + \frac{n\gamma}{n - 1} }\\
			&= \frac{\gamma + 2}{(n - 1)\gamma + n + \sqrt{\Big[(n - 1)\gamma + n \Big]^2 + \Big[(n - 1)  \gamma^2 + n\gamma  \Big](\gamma + 2)}}.
		\end{aligned}
	\end{equation*}
	If $\gamma = n - 4$, then
	\begin{equation*}
		\begin{aligned}
			S &= \frac{n - 2}{(n - 1)(n - 4) + n + \sqrt{\Big[(n - 1)(n - 4) + n \Big]^2 + \Big[(n - 1)  (n - 4) + n  \Big](n - 2)(n - 4)}}\\
			&= \frac{n - 2}{(n - 2)^2 + \sqrt{(n - 2)^4 + (n - 2)^{3}(n - 4)}}\\
			&= \frac{1}{n - 2 + \sqrt{(n - 2)(2n - 6)}}\\
			&= \frac{\sqrt{(n - 2)(2n - 6)} - n + 2}{(n - 2)(n - 4)}.
		\end{aligned}
	\end{equation*}
	So when $\varepsilon > 0$ is small enough, we have
	\begin{align*}
		S = \frac{\sqrt{(n - 2)(2n - 6)} - n + 2}{(n - 2)(n - 4)} + E(\varepsilon),
	\end{align*}
	where $E(\varepsilon)$ tends to $0$ when $\varepsilon$ tends to $0$.
	We hope the condition \eqref{condition08240930} holds. In the following computation, by the continuity of $\varepsilon$, we only need to check the conditions for $\varepsilon = 0$.
	
	\begin{lemma}\label{lemma1}
		If $n \geq 7, \gamma = n - 4, \varepsilon = 0, \alpha = -\gamma - \frac{4}{n} ,  p - P = \frac{n + 2}{n - 2} - \frac{1}{n^2}$, then
		\begin{equation}
			\mathbb K_1 > 0.
		\end{equation}
		
	\end{lemma}
	\begin{proof}[Proof of Lemma \ref{lemma1}]
	
	Recall that
	\begin{align*}
		a_1 &= \frac{1 + \gamma S }{1 - S^2 + \gamma S - \gamma S^2 - (n - 1)Q^2},\\
			a_2 &=\frac{1 + \gamma S }{1 - S^2 + \gamma S - \gamma S^2 - (n - 1)Q^2}(1 + \gamma) ,\\
			a_3 &= - \Big( \alpha + p\Big)\frac{\alpha - 1}{1 + \gamma S + 2S} + \frac{1 + \gamma S }{1 - S^2 + \gamma S - \gamma S^2 - (n - 1)Q^2}\cdot\frac{\alpha(\alpha - 1)(1 - S)}{1 + \gamma S + 2S} ,\\
			b_1 &= \frac{1 + \gamma S }{1 - S^2 + \gamma S - \gamma S^2 - (n - 1)Q^2}\cdot  \frac{\alpha(\gamma + 3)}{1 + \gamma S + 2S} - \Big( \alpha + p\Big)\frac{\gamma + 2}{1 + \gamma S + 2S} ,\\
			b_2 &=  \gamma  + \frac{1 + \gamma S }{1 - S^2 + \gamma S - \gamma S^2 - (n - 1)Q^2}\Big(2\gamma S - \gamma + 2S - 2Q\Big) ,\\
			c_2 &= - \frac{ q }{\gamma + q},\\
			b_3 &= 0,\\
			b_4 &= - p - \frac{ q\alpha }{\gamma + q} ,\\
			b_5 &= - \frac{ q}{\gamma + q},\\
			B_0 &= a_1\Big( \frac{n}{n - 1} + \gamma \Big) .
	\end{align*}
		Besides using $b_2 \equiv 0$ we have
		\begin{equation*}
			\begin{aligned}
				a_1 &= -\frac{\gamma}{2\gamma S - \gamma + 2S - 2Q}\\
				&= -\frac{\gamma}{2\gamma S - \gamma + 2S -  \frac{2}{n - 1}(1 - S)}\\
				&= -\frac{\gamma}{2\gamma S - \gamma + \frac{2n}{n - 1} S -  \frac{2}{n - 1}}\\
				&= -\frac{(n - 4)(n - 1)}{2S(n - 2)^2 - (n - 2)(n - 3)}\\
				&= -\frac{(n - 4)(n - 1)}{2\frac{\sqrt{(n - 2)(2n - 6)} - n + 2}{n - 4}(n - 2) - (n - 2)(n - 3)}\\
				&= \frac{2\sqrt{(n - 2)(2n - 6)} + n^2 - 5n + 8}{(n - 1)(n - 2)}.
			\end{aligned}
		\end{equation*}
	In order to prove $\mathbb K_1 > 0$ , we can simplify the computation and it is equivalent to prove that $4 a_1\left(\frac{n}{n - 1} + \gamma\right)a_3 - b_1^2 > 0$ with $p = \frac{n + 2}{n - 2} - \frac{1}{n^2}$.
		\begin{align*}
			 &\mathbb K_1 = 4a_1\left(\frac{n}{n - 1} + \gamma\right)a_3 - b_1^2\\
				&= 4a_1\left(\frac{n}{n - 1} + \gamma\right)\Bigg\{ -\frac{p(\alpha - 1)}{1 + \gamma S + 2S} + \Big[a_1(1 - S) - 1\Big]\frac{\alpha(\alpha - 1)}{1 + \gamma S + 2S} \Bigg\}\\
				&-\Bigg\{ -\frac{(\gamma + 2)p}{1 + \gamma S + 2S} + \frac{\alpha}{1 + \gamma S + 2S}\Big[(\gamma + 3)a_1 - (\gamma + 2)\Big] \Bigg\}^2\\
				&= 4a_1\left(\frac{n}{n - 1} + n - 4\right)\Bigg\{ -\frac{p(\alpha - 1)}{1 + (n - 2) S} + \Big[a_1(1 - S) - 1\Big]\frac{\alpha(\alpha - 1)}{1 + (n - 2) S} \Bigg\}\\
				&-\Bigg\{ -\frac{(n - 2)p}{1 + (n - 2) S} + \frac{\alpha}{1 + (n - 2) S}\Big[(n - 1)a_1 - n + 2\Big] \Bigg\}^2.
		\end{align*}
		As a result, we know
		\begin{equation*}
			\begin{aligned}
				&\mathbb K_1 > 0, \\
				\Leftrightarrow & H:=  4 \Big[1 + (n - 2)S\Big]a_1\left(\frac{n}{n - 1} + n - 4\right)\Bigg\{ -p(\alpha - 1) + \Big[a_1(1 - S) - 1\Big]\alpha(\alpha - 1) \Bigg\}\\
				&-\Bigg\{ -(n - 2)p + \alpha\Big[(n - 1)a_1 - n + 2\Big] \Bigg\}^2 > 0.
			\end{aligned}
		\end{equation*}
		Using the facts that
		\begin{itemize}
			\item
			\begin{equation*}
			\begin{aligned}
				1 + (\gamma + 2)S = \frac{\sqrt{(n - 2)(2n - 6)} -2}{n - 4}.
			\end{aligned}
		\end{equation*}
		\item
		\begin{align*}
			&(n - 1)a_1 - n + 2\\
			&= \frac{2\sqrt{(n - 2)(2n - 6)} + n^2 - 5n + 8}{n - 2} - n + 2\\
			&= \frac{2\sqrt{(n - 2)(2n - 6)} - n + 4}{n - 2}.
		\end{align*}
		
		\item
		
		\begin{align*}
			&a_1 (1 - S) - 1 \\
			&= a_1 - 1 - a_1 S\\
			&= \frac{2\sqrt{(n - 2)(2n - 6)} + n^2 - 5n + 8}{(n - 1)(n - 2)} - 1 \\
			&- \frac{2\sqrt{(n - 2)(2n - 6)} + n^2 - 5n + 8}{(n - 1)(n - 2)}\cdot  \frac{\sqrt{(n - 2)(2n - 6)} - n + 2}{(n - 2)(n - 4)}\\
			&= \frac{2\sqrt{(n - 2)(2n - 6)} - 2n + 6}{(n - 1)(n - 2)} \\
			&- \frac{2\sqrt{(n - 2)(2n - 6)} + n^2 - 5n + 8}{(n - 1)(n - 2)}\cdot  \frac{\sqrt{(n - 2)(2n - 6)} - n + 2}{(n - 2)(n - 4)},
		\end{align*}
		
		\end{itemize}

		\begin{align*}
			\Rightarrow &a_1 (1 - S) - 1 \\
			&= \frac{1}{(n - 1)(n - 2)^2(n - 4)}\Bigg\{ \Big[2\sqrt{(n - 2)(2n - 6)} - 2n + 6\Big](n^2 - 6n + 8)\\
			&- \Big[ 2\sqrt{(n - 2)(2n - 6)} + n^2 - 5n + 8\Big] \cdot \Big[ \sqrt{(n - 2)(2n - 6)} - n + 2 \Big]  \Bigg\}\\
			&= \frac{1}{(n - 1)(n - 2)^2(n - 4)}\Bigg\{ (n^2 - 5n + 8)(n - 2) - 2(n - 2)(2n - 6)\\
			&- (2n - 6)(n^2 - 6n + 8) + \Big[ 2n^2 - 12n + 16 + 2n - 4 - n^2 + 5n - 8 \Big] \sqrt{(n - 2)(2n - 6)} \Bigg\}\\
			&= \frac{1}{(n - 1)(n - 2)^2(n - 4)}\Bigg\{ -(n - 1)(n - 2)(n - 4) + ( n^2 - 5n + 4 ) \sqrt{(n - 2)(2n - 6)} \Bigg\}\\
			&= \frac{1}{(n - 2)^2}\Big[ -(n - 2) +  \sqrt{(n - 2)(2n - 6)} \Big].\\
		\end{align*}

		Therefore, we obtain:
		\begin{align*}
			H &= 4  \frac{\sqrt{(n - 2)(2n - 6)} -2}{n - 4}\cdot \frac{2\sqrt{(n - 2)(2n - 6)} + n^2 - 5n + 8}{(n - 1)^2}\cdot(n - 2)\Bigg\{ -p(\alpha - 1) \\
			&+ \frac{1}{(n - 2)^2}\Big[ -(n - 2) +  \sqrt{(n - 2)(2n - 6)} \Big]\alpha(\alpha - 1) \Bigg\}\\
				&-\Bigg\{ -(n - 2)p + \alpha\frac{2\sqrt{(n - 2)(2n - 6)} - n + 4}{n - 2} \Bigg\}^2.\\
		\end{align*}

		If $\alpha = - n + 4 - \frac{4}{n}=-\frac{(n-2)^2}{n}$, then we find that
		\begin{align*}
			H(p) &= 4  \frac{\sqrt{(n - 2)(2n - 6)} -2}{n - 4}\cdot \frac{2\sqrt{(n - 2)(2n - 6)} + n^2 - 5n + 8}{(n - 1)^2}\cdot(n - 2)\Bigg\{ -p \\
			&+ \frac{1}{(n - 2)^2}\Big[ -(n - 2) +  \sqrt{(n - 2)(2n - 6)} \Big]\left(-n + 4 - \frac{4}{n}\right) \Bigg\}\left(-n + 3 - \frac{4}{n}\right) \\
				&-\Bigg\{ -(n - 2)p + \left(-n + 4 - \frac{4}{n}\right)\frac{2\sqrt{(n - 2)(2n - 6)} - n + 4}{n - 2} \Bigg\}^2.
		\end{align*}
		\begin{claim}\label{claim10}
		When $n \geq 7$, we have
		\begin{equation*}
			\begin{aligned}
				 H\left(\frac{n + 2}{n - 2} - \frac{1}{n^2}\right) > 0.
			\end{aligned}
		\end{equation*}
		\end{claim}

	\end{proof}

	\begin{lemma}\label{lemmaU0}
		If $n \geq 7, 1 < p \leq \frac{n + 2}{n - 2}$ and $U_2 > U_1 > 0$ are two solutions of
		\begin{equation}\label{lemmaEqu1}
			\mathbb K_3 := c_2 + U - \frac{1}{4B_0}\Big(U\frac{ \gamma + q}{1 + \gamma S + qS}\Big)^2 = 0.
		\end{equation}
		Define
		\begin{equation}
			U_0 :=  \frac{2B_0 (1 + \gamma S + qS)^2}{(\gamma + q)^2}\left[1 - \frac{1}{\sqrt{2}} \left(1 - \frac{2}{n}\right)\right],
		\end{equation}
		then when $U = U_0$ and $p - P = \frac{n + 2}{n - 2} - \frac{1}{n^2}$, we have
		\begin{equation}\label{lemma2Conclu1}
			U_1 < U_0 < U_2,
		\end{equation}
		and
		\begin{equation}
			\begin{aligned}
				&\mathbb K_3, \mathbb K_5 > 0.
			\end{aligned}
		\end{equation}
	\end{lemma}
	If Lemma \ref{lemmaU0} holds and choose $0<T <<|P|$ small enough, then we get $\mathbb K_i - \varepsilon > 0$ for $i = 1, 2, 3, 5, 6$.
	Then for $1 < p < \frac{n + 2}{n - 2} - \frac{1}{n^2}$, we get that
	\begin{align*}
		P = p - \frac{n + 2}{n - 2} + \frac{1}{n^2} < 0.
	\end{align*}
	As a result, the condition \eqref{condition08240930} holds and we deduce that $|\nabla v| \equiv 0$. Therefore we only need to consider the case that $\frac{n + 2}{n - 2} - \frac{1}{n^2} \leq p \leq \frac{n + 2}{n - 2}$. At this time, $P \geq 0$ and we will give the lower bound of $M$.
%

	\begin{proof}[Proof of Lemma \ref{lemmaU0} ]
		Firstly, we can rewrite \eqref{lemmaEqu1} as:
		\begin{equation}
			\begin{aligned}
				&-\frac{q}{\gamma + q} + U - \frac{1}{4a_1\Big( \frac{n}{n - 1} + \gamma \Big) }\Big(U\frac{ \gamma + q}{1 + \gamma S + qS}\Big)^2 = 0.\\
			\end{aligned}
		\end{equation}
	We know that for the right hand of above identity with respect to $U$
	
	\begin{align*}
			\Delta &= 1 - \frac{q(\gamma + q)}{a_1 (\frac{n}{n - 1} + \gamma )(1 + \gamma S + qS)^2}\\
			&> 1 - \frac{q(\gamma + q)}{a_1 (\frac{n}{n - 1} + \gamma )(1 + \gamma S )^2}\\
			&\geq 1 - \frac{(1 + \frac{2}{n})(\gamma + 1 + \frac{2}{n})}{a_1 (\frac{n}{n - 1} + \gamma )(1 + \gamma S )^2}\\
			&= 1 - \frac{(1 + \frac{2}{n})(n - 3 + \frac{2}{n})(n - 2)^2(n - 1)}{ (\frac{n}{n - 1} + n - 4 )(2n - 6)\Big[ 2\sqrt{(n - 2)(2n - 6)} + n^2 - 5n + 8\Big] }.
		\end{align*}
		When $n = 7$, we have $\Delta > 0.284  > 0.5 - \frac{2}{n} + \frac{2}{n^2}$; when $n = 8$, we have $\Delta > 0.322 > 0.5 - \frac{2}{n} + \frac{2}{n^2}$.
		\begin{claim}\label{claim1}
		When $n \geq 9$, then
			\begin{align*}
				\Delta > 0.5 - \frac{2}{n} + \frac{2}{n^2} > 0.
			\end{align*}
		\end{claim}
		Thus there always exists two solutions $U_{2} > U_{1} > 0$. Define
		
		\begin{align*}
			U_0 &:= \frac{2B_0 (1 + \gamma S + qS)^2}{(\gamma + q)^2}\left[1 - \frac{1}{\sqrt{2}} \left(1 - \frac{2}{n}\right)\right]\\
			&= \frac{2B_0 (1 + \gamma S + qS)^2}{(\gamma + q)^2}\left( \frac{2 - \sqrt{2}}{2} + \frac{\sqrt{2}}{n} \right),
		\end{align*}
		then $\mathbb K_3(U_0) > 0$  and $U_1 < U_0 < U_2$. Finally we will check that $\mathbb K_5 > 0$ when $U = U_0$.
		Define
		\begin{align*}
			I(U, p - P) &:= - \mathbb K_5 = b_4 + P +  \frac{U\alpha}{1 + \gamma S + qS}  - \frac{1}{2B_0} \Big( b_1 + P  \frac{\gamma + 2}{1 + \gamma S + 2S}\Big) U\frac{ \gamma + q}{1 + \gamma S + qS}\\
			&= - (p - P) - \frac{q\alpha}{\gamma + q} +  \frac{U\alpha}{1 + \gamma S + qS} \\
			& - \frac{1}{2B_0} \Big[ a_1\cdot  \frac{\alpha(\gamma + 3)}{1 + \gamma S + 2S} - \Big( \alpha + p - P\Big)\frac{\gamma + 2}{1 + \gamma S + 2S}\Big] U\frac{ \gamma + q}{1 + \gamma S + qS}\\
			&= \left[ \frac{U(\gamma + 2)(\gamma + q)}{2B_0(1 + \gamma S + 2S)(1 + \gamma S + qS)} - 1\right](p - P) - \frac{q\alpha}{\gamma + q} +  \frac{U\alpha}{1 + \gamma S + qS}  \\
			& - \frac{1}{2B_0} \Big[ a_1\cdot  \frac{\alpha(\gamma + 3)}{1 + \gamma S + 2S} - \frac{\alpha( \gamma + 2)}{1 + \gamma S + 2S}\Big] U\frac{ \gamma + q}{1 + \gamma S + qS}.
		\end{align*}
		If $p - P = \frac{n + 2}{n - 2} - \frac{1}{n^2}$,
		\begin{align*}
			I\left(U_{0}, \frac{n + 2}{n - 2} - \frac{1}{n^2}\right) &= \left[ \left( \frac{2 - \sqrt{2}}{2} + \frac{\sqrt{2}}{n} \right) \frac{(\gamma + 2)(1 + \gamma S + qS)}{(\gamma + q)(1 + \gamma S + 2S)} - 1\right]\left(\frac{n + 2}{n - 2} - \frac{1}{n^2}\right) - \frac{q\alpha}{\gamma + q} \\
			&+ \left( \frac{2 - \sqrt{2}}{2} + \frac{\sqrt{2}}{n} \right) \frac{2B_0 (1 + \gamma S + qS)\alpha}{(\gamma + q)^2}\\
			& -  \Big[ a_1\cdot  \frac{\alpha(\gamma + 3)}{1 + \gamma S + 2S} - \frac{\alpha( \gamma + 2)}{1 + \gamma S + 2S}\Big] \frac{1 + \gamma S + qS}{\gamma + q}\left( \frac{2 - \sqrt{2}}{2} + \frac{\sqrt{2}}{n} \right)\\
			&=  \left[ \left( \frac{2 - \sqrt{2}}{2} + \frac{\sqrt{2}}{n} \right)\underset{\mathbb D_1 }{ \frac{(\gamma + 2)(1 + \gamma S + qS)}{(\gamma + q)(1 + \gamma S + 2S)}} - 1\right]\left(\frac{n + 2}{n - 2} - \frac{1}{n^2}\right) \\
			&+ \frac{\alpha}{\gamma + q}\Bigg[- q + \left( \frac{2 - \sqrt{2}}{2} + \frac{\sqrt{2}}{n} \right) \underset{\mathbb D_2}{\frac{2B_0 (1 + \gamma S + qS)}{\gamma + q}}\\
			& - \underset{\mathbb D_3}{ \frac{2\sqrt{(n - 2)(2n - 6)} - n + 4}{n - 2}\cdot \frac{ 1 + \gamma S + qS}{1 + \gamma S + 2S}\left( \frac{2 - \sqrt{2}}{2} + \frac{\sqrt{2}}{n} \right)}\Bigg].
		\end{align*}
		Then we need to estimate $\mathbb D_1$, $\mathbb D_2$ and $\mathbb D_3$.
		Using the fact that if $x = \frac{1}{n - 2} \leq \frac{1}{5}$,
		\begin{equation*}
			\begin{aligned}
				& 1 - \frac{1}{2}x - \frac{1}{6}x^2 < \sqrt{1 - x} < 1 - \frac{1}{2}x - \frac{1}{8}x^2,\\
			 \Rightarrow& \sqrt{2}\Big[ n - \frac{5}{2} - \frac{1}{6(n - 2)}\Big] < \sqrt{(n - 2)(2n - 6)} = \sqrt{2}(n - 2)\sqrt{1 - \frac{1}{n - 2}} < \sqrt{2}\Big[ n - \frac{5}{2} - \frac{1}{8(n - 2)}\Big],
			\end{aligned}
		\end{equation*}
	 we get the following estimates:
	 \begin{itemize}
	 	\item $\mathbb D_1$:
	 	 Firstly, we know that
	 	 \begin{align*}
			\frac{(\gamma + 2)(1 + \gamma S + qS)}{(\gamma + q)(1 + \gamma S + 2S)} - 1 &= \frac{ 2 - q}{(\gamma + q)(1 + \gamma S + 2S)} \\
			&= \frac{2 - q}{n - 4 + q}\cdot \frac{n - 4}{\sqrt{(n - 2)(2n - 6)} - 2}\\
			&< \frac{2 - q}{n - 4 + q}\cdot \frac{n - 4}{\sqrt{2}(n - 2)\Big[ 1 - \frac{1}{2(n - 2)} - \frac{1}{6(n - 2)^2}\Big] - 2} \\
			&= \frac{2 - q}{n - 4 + q}\cdot \frac{n - 4}{\sqrt{2}\Big[ n - \frac{5}{2} - \frac{1}{6(n - 2)} - \sqrt{2}\Big] },\\
		\end{align*}
		and
		\begin{align*}
				&\left(-\frac{1}{2} + \frac{1}{n}\right)\left(\frac{n + 2}{n - 2} - \frac{1}{n^2}\right) \\
				&= -\frac{1}{2} - \frac{1}{n} + \frac{n - 2}{2n^3} <  -\frac{1}{2} - \frac{1}{n} + \frac{1}{2n^2}.\\
			\end{align*}
	 Besides, we have the following estimates:
		\begin{claim}\label{claim2}
			When $n \geq 7$, we have
			\begin{align*}
				\left( \frac{2 - \sqrt{2}}{2} + \frac{\sqrt{2}}{n} \right) \cdot \frac{n - 4}{\sqrt{2}\Big( n - \frac{5}{2} - \frac{1}{6(n - 2)}  - \sqrt{2}\Big) }\left(\frac{n + 2}{n - 2} - \frac{1}{n^2}\right) < \frac{\sqrt{2} - 1}{2} + \frac{3}{n}.
			\end{align*}
		\end{claim}		
		Then we get that
		\begin{align*}
			&\left[ \left( \frac{2 - \sqrt{2}}{2} + \frac{\sqrt{2}}{n} \right)\underset{\mathbb D_1}{ \frac{(\gamma + 2)(1 + \gamma S + qS)}{(\gamma + q)(1 + \gamma S + 2S)}} - 1\right]\left(\frac{n + 2}{n - 2} - \frac{1}{n^2}\right) \\
			&<\left\{ \left( \frac{2 - \sqrt{2}}{2} + \frac{\sqrt{2}}{n} \right) \left[1 + \frac{2 - q}{n - 4 + q}\cdot \frac{n - 4}{\sqrt{2}\Big( n - \frac{5}{2} - \frac{1}{6(n - 2)}  - \sqrt{2}\Big) }\right] - 1\right\}\left(\frac{n + 2}{n - 2} - \frac{1}{n^2}\right)\\
			&=   \left( \frac{2 - \sqrt{2}}{2} + \frac{\sqrt{2}}{n} \right) \frac{2 - q}{n - 4 + q}\cdot \frac{n - 4}{\sqrt{2}\Big( n - \frac{5}{2} - \frac{1}{6(n - 2)}  - \sqrt{2}\Big) }\left(\frac{n + 2}{n - 2} - \frac{1}{n^2}\right) \\
			&+ \left( -\frac{\sqrt{2}}{2} + \frac{\sqrt{2}}{n} \right)\left(\frac{n + 2}{n - 2} - \frac{1}{n^2}\right)\\
			&< -\frac{\sqrt{2}}{2} - \frac{\sqrt{2}}{n} + \frac{\sqrt{2}}{2n^2} + \frac{2 - q}{n - 4 + q}\Bigg( \frac{\sqrt{2} - 1}{2} + \frac{3}{n}  \Bigg).
		\end{align*}

		\item $\mathbb D_2$:
		
		\begin{align*}
			&\frac{2B_0 (1 + \gamma S + qS)}{\gamma + q} \\
			&= 2\left(\gamma + \frac{n}{n - 1}\right)a_1\left( \frac{1}{\gamma + q} + S\right)\\
			&= 2\left(n - 3 + \frac{1}{n - 1}\right)\frac{2\sqrt{(n - 2)(2n - 6)} + n^2 - 5n + 8}{(n - 1)(n - 2)}\left[ \frac{1}{n - 4 + q} + \frac{\sqrt{(n - 2)(2n - 6)} - n + 2}{(n - 2)(n - 4)}\right]\\
			&> 2\left(n - 3 + \frac{1}{n - 1}\right)\frac{2\sqrt{2}\Big[ n - \frac{5}{2} - \frac{1}{6(n - 2)}\Big] + n^2 - 5n + 8}{(n - 1)(n - 2)}\left[ \frac{1}{n - 4 + q} + \frac{\sqrt{2}\Big[ n - \frac{5}{2} - \frac{1}{6(n - 2)}\Big] - n + 2}{(n - 2)(n - 4)}\right]\\
			&\geq 2\left(n - 3 + \frac{1}{n - 1}\right)\frac{2\sqrt{2}\Big( n - \frac{38}{15} \Big) + n^2 - 5n + 8}{(n - 1)(n - 2)}\left[ \frac{1}{n - 4 + q} + \frac{\sqrt{2}\Big( n - \frac{38}{15} \Big) - n + 2}{(n - 2)(n - 4)}\right].\\
		\end{align*}
		Using the fact that when $n = 7$,
		\begin{align*}
			\frac{1}{n} + \frac{1}{n(n - 1)} = \frac{0.5}{n - 4},
		\end{align*}
		and when $n \geq 8$
		\begin{align*}
			\frac{1}{n} \geq \frac{0.5}{n - 4},
		\end{align*}
		we get that
		\begin{align*}
			\frac{1}{n - 4 + q} &\geq \frac{1}{n - 3 + \frac{2}{n}} = \frac{1}{n - 2}\cdot \frac{1}{1 - \frac{1}{n}} \\
			&= \frac{1}{n - 2}\cdot \left[1 + \frac{1}{n} + \frac{1}{n(n - 1)}\right] \\
			&\geq \frac{1}{n - 2}\cdot \left( 1 + \frac{0.5}{n - 4}\right) \\
			 &= \frac{1}{n - 2} + \frac{0.5}{(n - 2)(n - 4)},
		\end{align*}
		we get that:
		\begin{align*}
			 &\frac{2B_0 (1 + \gamma S + qS)}{\gamma + q} \\
			&> 2\left(n - 3 + \frac{1}{n - 1}\right)\frac{2\sqrt{2}\Big( n - \frac{38}{15} \Big) + n^2 - 5n + 8}{(n - 1)(n - 2)}\left[ \frac{1}{n - 2} + \frac{0.5}{(n - 2)(n - 4)} + \frac{\sqrt{2}\Big( n - \frac{38}{15} \Big) - n + 2}{(n - 2)(n - 4)}\right]\\
			&\geq 2\left(n - 3 + \frac{1}{n - 1}\right)\frac{2\sqrt{2}\Big( n - \frac{38}{15} \Big) + n^2 - 5n + 8}{(n - 1)(n - 2)^2(n - 4)}\left[ n - 3.5 + \sqrt{2}\Big( n - \frac{38}{15} \Big) - n + 2\right]\\
			&= 2\left(n - 3 + \frac{1}{n - 1}\right)\frac{2\sqrt{2}\Big( n - \frac{38}{15} \Big) + n^2 - 5n + 8}{(n - 1)(n - 2)^2(n - 4)}\left[ \sqrt{2}\Big( n - \frac{53}{21} \Big) -  1.5\right].
		\end{align*}
		\begin{claim}\label{claim3}
			When $n \geq 7$, we have
			\begin{align*}
				&2\left(n - 3 + \frac{1}{n - 1}\right)\frac{2\sqrt{2}\Big( n - \frac{38}{15} \Big) + n^2 - 5n + 8}{(n - 1)(n - 2)^2(n - 4)}\left[ \sqrt{2}\Big( n - \frac{38}{15} \Big) -  1.5\right]\\
			&> 2\sqrt{2} + \frac{0.4\sqrt{2}}{n}.
			\end{align*}
		\end{claim}		
		Thus we obtain that
		\begin{align*}
			 &\frac{2B_0 (1 + \gamma S + qS)}{\gamma + q} > 2\sqrt{2} + \frac{0.4\sqrt{2}}{n}.
		\end{align*}
		\item $\mathbb D_3$:
		\begin{align*}
			\frac{1 + \gamma S + qS}{1 + \gamma S + 2S} - 1 &= -\frac{(2 - q)S}{1 + \gamma S + 2S}\\
			&= -\frac{2 - q}{n - 2}\cdot\frac{\sqrt{(n - 2)(2n - 6)} - n + 2}{\sqrt{(n - 2)(2n - 6)} - 2}\\
			&< -\frac{2 - q}{n - 2}\cdot \frac{ \sqrt{2}\Big[ n - \frac{5}{2} - \frac{1}{6(n - 2)}\Big] - n + 2}{ \sqrt{2}\Big[ n - \frac{5}{2} - \frac{1}{8(n - 2)}\Big] - 2}\\
			&<  -\frac{2 - q}{n - 2}\cdot \frac{ \frac{2 - \sqrt{2}}{2} n  - \frac{5}{2} - \frac{1}{6(n - 2)}  + \sqrt{2}}{n - \frac{5}{2} - \sqrt{2}}.\\
		\end{align*}
		
		\begin{align*}
			& \frac{2\sqrt{(n - 2)(2n - 6)} - n + 4}{n - 2}\cdot \frac{ 1 + \gamma S + qS}{1 + \gamma S + 2S}\left( \frac{2 - \sqrt{2}}{2} + \frac{\sqrt{2}}{n} \right)\\
			&< \frac{2\sqrt{(n - 2)(2n - 6)} - n + 4}{n - 2}\cdot \left( 1 -\frac{2 - q}{n - 2}\cdot \frac{ \frac{2 - \sqrt{2}}{2} n  - \frac{5}{2} - \frac{1}{6(n - 2)}  + \sqrt{2}}{n - \frac{5}{2} - \sqrt{2}} \right)\left( \frac{2 - \sqrt{2}}{2} + \frac{\sqrt{2}}{n} \right)\\
			&= \frac{2\sqrt{(n - 2)(2n - 6)} - n + 4}{n - 2}\cdot \left( \frac{2 - \sqrt{2}}{2} + \frac{\sqrt{2}}{n} \right) \\
			&- \frac{2 - q}{n - 2}\cdot \frac{ \frac{2 - \sqrt{2}}{2} n  - \frac{5}{2} - \frac{1}{6(n - 2)}  + \sqrt{2}}{n - \frac{5}{2} - \sqrt{2}}\cdot\frac{2\sqrt{(n - 2)(2n - 6)} - n + 4}{n - 2}\left( \frac{2 - \sqrt{2}}{2} + \frac{\sqrt{2}}{n} \right).
		\end{align*}
		
		\begin{claim}\label{claim4}
			When $n \geq 7$, we have
			\begin{align*}
				 \frac{2\sqrt{(n - 2)(2n - 6)} - n + 4}{n - 2}\cdot \left( \frac{2 - \sqrt{2}}{2} + \frac{\sqrt{2}}{n} \right) < (2\sqrt{2} - 1)\cdot  \frac{2 - \sqrt{2}}{2} + \frac{2.76}{n} + \frac{1.7}{n^2}.\\
			\end{align*}
			
		\end{claim}
				
		\begin{claim}\label{claim5}
			When $n \geq 7$, we have
			\begin{align*}
				&\frac{1}{n - 2}\cdot \frac{ \frac{2 - \sqrt{2}}{2} n  - \frac{5}{2} - \frac{1}{6(n - 2)}  + \sqrt{2}}{n - \frac{5}{2} - \sqrt{2}}\cdot\frac{2\sqrt{(n - 2)(2n - 6)} - n + 4}{n - 2}\left( \frac{2 - \sqrt{2}}{2} + \frac{\sqrt{2}}{n} \right)\\
				&> \frac{8\sqrt{2} - 11}{2}\left( \frac{1}{n} + \frac{7}{n^2}\right) .
			\end{align*}
		\end{claim}
		Hence we get that
		\begin{align*}
			& \frac{2\sqrt{(n - 2)(2n - 6)} - n + 4}{n - 2}\cdot \frac{ 1 + \gamma S + qS}{1 + \gamma S + 2S}\left( \frac{2 - \sqrt{2}}{2} + \frac{\sqrt{2}}{n} \right)\\
			&< (2\sqrt{2} - 1)\cdot  \frac{2 - \sqrt{2}}{2} + \frac{2.76}{n} + \frac{1.7}{n^2}\\
			& - (2 - q)\frac{8\sqrt{2} - 11}{2}\left( \frac{1}{n} + \frac{7}{n^2}\right) .
		\end{align*}

	 \end{itemize}

		By direct computation, we get that
		\begin{align*}
			\lim_{n \rightarrow \infty} I\left(U_{0}, \frac{n + 2}{n - 2} - \frac{1}{n^2}\right) = 0.
		\end{align*}
		So we should pay more attention to the terms $O(\frac{1}{n}). $
		\begin{align*}
			\Rightarrow &I\left(U_{0}, \frac{n + 2}{n - 2} - \frac{1}{n^2}\right) \\
			&\leq -\frac{\sqrt{2}}{2} - \frac{\sqrt{2}}{n} + \frac{\sqrt{2}}{2n^2} + \frac{2 - q}{n - 4 + q}\Bigg( \frac{\sqrt{2} - 1}{2} + \frac{3}{n}  \Bigg) \\
			&+ \frac{- n + 4 - \frac{4}{n}}{n - 4 + q}\Bigg\{ -q + \left( \frac{2 - \sqrt{2}}{2} + \frac{\sqrt{2}}{n} \right) \cdot \Bigg(  2\sqrt{2} + \frac{0.4\sqrt{2}}{n} \Bigg)\\
			&-(2\sqrt{2} - 1)\cdot  \frac{2 - \sqrt{2}}{2} - \frac{2.76}{n} - \frac{1.7}{n^2} + (2 - q)\frac{8\sqrt{2} - 11}{2}\left( \frac{1}{n} + \frac{7}{n^2}\right)\Bigg\}\\
			&\leq -\frac{\sqrt{2}}{2} - \frac{\sqrt{2}}{n} + \frac{\sqrt{2}}{2n^2} + \frac{2 - q}{n - 3}\Bigg( \frac{\sqrt{2} - 1}{2} + \frac{3}{n}  \Bigg)  \\
			&- \frac{ n - 4 + \frac{4}{n}}{n - 4 + q}\Bigg\{ -q + \left( \frac{2 - \sqrt{2}}{2} + \frac{\sqrt{2}}{n} \right) \cdot \Bigg(  2\sqrt{2} + \frac{0.4\sqrt{2}}{n} \Bigg)\\
			&+ 3 - 2.5\sqrt{2} - \frac{2.76}{n} - \frac{1.7}{n^2} + (2 - q)\frac{8\sqrt{2} - 11}{2}\left( \frac{1}{n} + \frac{7}{n^2}\right) \Bigg\},\\	
		\end{align*}
		Thus for $n \geq 7$, we have:
		\begin{equation}\label{n>7}
			\begin{aligned}
				\Rightarrow &I\left(U_{0}, \frac{n + 2}{n - 2} - \frac{1}{n^2}\right) \\
			&<  -\frac{\sqrt{2}}{2} - \frac{\sqrt{2}}{n} + \frac{\sqrt{2}}{2n^2} + \frac{2 - q}{n - 3}\Bigg( \frac{\sqrt{2} - 1}{2} + \frac{3}{n}  \Bigg)  \\
			&- \frac{ n - 4 + \frac{4}{n}}{n - 4 + q}\Bigg[ 1 - q - 0.5\sqrt{2} + \frac{0.84 + 0.4\sqrt{2}}{n} - \frac{0.9}{n^2}   + (2 - q)\frac{8\sqrt{2} - 11}{2}\left( \frac{1}{n} + \frac{7}{n^2}\right)   \Bigg].\\
			\end{aligned}
		\end{equation}

		\begin{itemize}
			
			\item
			When $n \geq 9$, we have
		\begin{align*}
				&1 - q - 0.5\sqrt{2} + \frac{0.84 + 0.4\sqrt{2}}{n} - \frac{0.9}{n^2} + (2 - q)\frac{8\sqrt{2} - 11}{2}\left( \frac{1}{n} + \frac{7}{n^2}\right) \\
				&< - 0.5\sqrt{2} + \frac{0.84 + 0.4\sqrt{2}}{n} - \frac{0.9}{n^2}  + \frac{8\sqrt{2} - 11}{2}\left( \frac{1}{n} + \frac{7}{n^2}\right) \\
				&= - 0.5\sqrt{2} + \frac{4.4\sqrt{2} - 4.66}{n} + \frac{28\sqrt{2} - 39.4}{n^2}   \\
				&\leq - 0.5\sqrt{2} + \frac{4.4\sqrt{2} - 4.66}{7} + \frac{28\sqrt{2} - 39.4}{49}  < 0 . \\
			\end{align*}
			
			\item When $n\geq 9$
			\begin{equation*}
				\begin{aligned}
					\frac{1}{n - 3} = \frac{1}{n}\left(1 + \frac{3}{n} + \frac{9}{n^2}\cdot\frac{1}{1 - \frac{3}{n}}\right) \leq \frac{1}{n}\left(1 + \frac{3}{n} + \frac{1}{n}\cdot\frac{9}{6}\right) = \frac{1}{n} + \frac{4.5}{n^2}.
				\end{aligned}
			\end{equation*}

		\end{itemize}
		\begin{claim}\label{claim6}
			When $n \geq 9$, we have
		
		\begin{align*}
				\frac{ n - 4 + \frac{4}{n}}{n - 4 + q} &<  1 - \frac{q}{n} + \frac{1.5q^2 - 8q + 12}{n^2}.
			\end{align*}

		\end{claim}
		
		Noting that when $n \geq 9$ and $1 < q \leq 1 + \frac{2}{n}\leq \frac{11}{9}$, we have that
		\begin{align*}
			\Rightarrow &I\left(U_{0}, \frac{n + 2}{n - 2} - \frac{1}{n^2}\right) \\
			&\leq  -\frac{\sqrt{2}}{2} - \frac{\sqrt{2}}{n} + \frac{\sqrt{2}}{2n^2} + (2 - q)\left( \frac{1}{n} + \frac{4.5}{n^2} \right)\Bigg( \frac{\sqrt{2} - 1}{2} + \frac{3}{n}  \Bigg)  \\
			&- \left( 1 - \frac{q}{n} + \frac{1.5q^2 - 8q + 12}{n^2}\right)\Bigg[1 - q - 0.5\sqrt{2} + \frac{0.84 + 0.4\sqrt{2}}{n} - \frac{0.9}{n^2}  \\
			&+ (2 - q)\frac{8\sqrt{2} - 11}{2}\left( \frac{1}{n} + \frac{7}{n^2}\right)  \Bigg]\\
			&= \Bigg[ -\sqrt{2} + \frac{\sqrt{2} - 1}{2}(2 - q) + n(q - 1) - 0.84 - 0.4\sqrt{2} -  (2 - q)\frac{8\sqrt{2} - 11}{2} - 0.5\sqrt{2}q \Bigg] \frac{1}{n}\\
			&+ \Bigg[ \frac{\sqrt{2}}{2} + 3(2 - q) + \frac{9\sqrt{2} - 9}{4}(2 - q) + 0.9 - (2 - q)\frac{56\sqrt{2} - 77}{2} + nq(1 - q) \\
			&+ q(0.84 + 0.4\sqrt{2}) + q(2 - q)\frac{8\sqrt{2} - 11}{2} + 0.5\sqrt{2}(1.5q^2 - 8q + 12) \Bigg] \frac{1}{n^2}\\
			&+ \Bigg[ 13.5(2 - q) - 0.9q + q(2 - q)\frac{56\sqrt{2} - 77}{2} \Bigg] \frac{1}{n^3}\\
			&- \frac{1.5q^2 - 8q + 12}{n^2}\Bigg[ 1 - q  + \frac{0.84 + 0.4\sqrt{2}}{n} - \frac{0.9}{n^2} + (2 - q)\frac{8\sqrt{2} - 11}{2}\left( \frac{1}{n} + \frac{7}{n^2}\right)  \Bigg]\\
		\end{align*}
		\begin{equation*}
			\begin{aligned}
				\Rightarrow &I\left(U_{0}, \frac{n + 2}{n - 2} - \frac{1}{n^2}\right) \\
			&< \Bigg[ -2.8 + 0.051(2 - q) + n(q - 1)   - 0.5\sqrt{2}q \Bigg] \frac{1}{n}\\
			&+ \Bigg[ 1.61 + 2.84(2 - q)  + nq(1 - q)+ 1.41q + q(2 - q)\frac{8\sqrt{2} - 11}{2} + 0.5\sqrt{2}(1.5 - 8 + 12) \Bigg] \frac{1}{n^2}\\
			&+ \Bigg[ 13.5(2 - q) - 0.9q + q(2 - q)\frac{56\sqrt{2} - 77}{2} \Bigg] \frac{1}{n^3}\\
			&- \frac{1.5q^2 - 8q + 12}{n^2}\Bigg[ -\frac{2}{n}  + \frac{0.84 + 0.4\sqrt{2}}{n} - \frac{0.9}{n^2} + \left(1 - \frac{2}{n}\right)\frac{8\sqrt{2} - 11}{2}\left( \frac{1}{n} + \frac{7}{n^2}\right)  \Bigg]
			\end{aligned}
		\end{equation*}
		\begin{align*}
			\Rightarrow &I\left(U_{0}, \frac{n + 2}{n - 2} - \frac{1}{n^2}\right) \\
			&< \Bigg[ -2.8 + 0.051(2 - q) + n(q - 1)   - 0.5\sqrt{2}q \Bigg] \frac{1}{n}\\
			&+ \Bigg[ 5.5 + 2.84(2 - q)  + nq(1 - q)+ 1.41q + q(2 - q)\frac{8\sqrt{2} - 11}{2}  \Bigg] \frac{1}{n^2}\\
			&+ \Bigg[ 13.5(2 - q) - 0.9q + q(2 - q)\frac{56\sqrt{2} - 77}{2} \Bigg] \frac{1}{n^3}\\
			&- \frac{1.5q^2 - 8q + 12}{n^2}\Bigg( -\frac{0.44}{n}  - \frac{0.12}{n^2} - \frac{2.2}{n^3}   \Bigg)\\
			&< \Bigg[ -2.8 + 0.051(2 - q) + n(q - 1)   - 0.5\sqrt{2}q \Bigg] \frac{1}{n}\\
			&+ \Bigg[ 5.5 + 2.84(2 - q)  + nq(1 - q)+ 1.41q + q(2 - q)\frac{8\sqrt{2} - 11}{2}  \Bigg] \frac{1}{n^2}\\
			&+ \Bigg[ 13.5(2 - q) - 0.9 + \frac{56\sqrt{2} - 77}{2} \Bigg] \frac{1}{n^3} + \frac{5.5}{n^2}\Bigg( \frac{0.44}{n} + \frac{0.12}{9n} + \frac{2.2}{81n}   \Bigg)\\
		\end{align*}
		\begin{align*}
			\Rightarrow &I\left(U_{0}, \frac{n + 2}{n - 2} - \frac{1}{n^2}\right) \\
			&< \Bigg[ -2.8 + 0.051(2 - q) + n(q - 1)   - 0.5\sqrt{2}q \Bigg] \frac{1}{n}\\
			&+ \Bigg[ 5.5 + 2.84(2 - q)  + nq(1 - q)+ 1.41q + q(2 - q)\frac{8\sqrt{2} - 11}{2}  \Bigg] \frac{1}{n^2} +  \frac{13.5(2 - q) + 2.85}{n^3} \\
			&< \Bigg[ -2.8 + 0.051(2 - q) + n(q - 1)   - 0.5\sqrt{2}q \Bigg] \frac{1}{n}\\
			&+ \Bigg[ 5.5 + 2.84(2 - q)  + nq(1 - q)+ 1.41q + \frac{8\sqrt{2} - 11}{2}  \Bigg] \frac{1}{n^2} +  \frac{13.5(2 - q) + 2.85}{9n^2} .\\
		\end{align*}

		\begin{equation}\label{Theorem_equ1}
			\begin{aligned}
			\Rightarrow &I\left(U_{0}, \frac{n + 2}{n - 2} - \frac{1}{n^2}\right) \\
			&< \Bigg[ -2.8 + 0.051(2 - q) + n(q - 1)   - 0.5\sqrt{2}q \Bigg] \frac{1}{n}\\
			&+ \Bigg[ 6 + 4.34(2 - q)  + nq(1 - q)+ 1.41q \Bigg] \frac{1}{n^2}.\\
			\end{aligned}
		\end{equation}

		\begin{claim}\label{claim 9}
			When $n =7, 8$, the inequality \eqref{Theorem_equ1} still holds.
		\end{claim}

		Since $n\geq 7$ and $1 < q \leq 1 + \frac{2}{n}$, we know
		\begin{equation*}
			\begin{aligned}
				  &-\mathbb K_5 = I\left(U_{0}, \frac{n + 2}{n - 2} - \frac{1}{n^2}\right) \\
				&< \Bigg[ -2.8 + 0.051(2 - q) + n(q - 1)   - 0.5\sqrt{2}q \Bigg] \frac{1}{n} \\
			&+ \Bigg[ 6 + 4.34(2 - q) + 1.41q \Bigg]\frac{1}{n^2}\\
			&< \left(n - 0.758 - \frac{2.93}{n}\right)\frac{q}{n} + \frac{-n - 2.698 }{n} + \frac{14.68}{n^2}\\
			&\leq \left(n - 0.758 - \frac{2.93}{n}\right)\frac{n + 2}{n^2} + \frac{-n - 2.698 }{n} + \frac{14.68}{n^2}\\
			&= \frac{- 0.758n^2 - 2.93n + 2n^2 - 1.516n - 5.86 - 2.698 n^2 + 14.68n}{n^3}\\
			&= \frac{- 1.456n^2 + 10.234n - 5.86 }{n^3} < 0.
			\end{aligned}
		\end{equation*}

	\end{proof}

	\begin{lemma}
		If $1 < p \leq \frac{n + 2}{n - 2}, q = \frac{2p}{p + 1}$, then there exists $0 \leq M_2 < M_1$, where $M_1$ comes from Theorem \ref{Theorem1}, such that for any $M > M_2$, the condition \eqref{condition08240930} is satisfied.
	\end{lemma}
	\begin{proof}
		By Lemma \ref{lemmaU0}, we know when $1 < p < \frac{n + 2}{n - 2} - \frac{1}{n^2}$, we can choose
	\begin{equation}
		\begin{aligned}
			 &P = p - \frac{n + 2}{n - 2} + \frac{1}{n^2} < 0,\\
			 &0 < T <<|P|.
		\end{aligned}
	\end{equation}
	Then we get $\mathbb K_4 > 0$ and the condition \eqref{condition08240930} holds. At this time, we can choose $M_2 := 0 < M_1$.
	
	Now let us consider the case $\frac{n + 2}{n - 2} - \frac{1}{n^2}\leq p \leq \frac{n + 2}{n - 2}$.
	We choose
	\begin{equation}
		\begin{aligned}
			 &P = p - \frac{n + 2}{n - 2} + \frac{1}{n^2} + \varepsilon > 0 ,\\
			 &0 < T <<|P|.
		\end{aligned}
	\end{equation}
	This case $\mathbb K_4 = -P + O(\varepsilon^2) < 0$ and we will use $v^{\alpha + p}|\nabla v|^{\gamma + q}$ and $v^{\alpha - 1}|\nabla v|^{\gamma + q + 2}$ to control the term $v^{\alpha + p - 1}|\nabla v|^{\gamma + 2}$. By Young inequality, we have
	\begin{equation}
		\begin{aligned}
			v^{\alpha + p - 1}|\nabla v|^{\gamma + 2}\leq \frac{q}{2}\Big(JM\cdot\frac{2}{2 - q}\Big)^{-\frac{2 - q}{q}}  v^{\alpha + p}|\nabla v|^{\gamma + q} + J M v^{\alpha - 1}|\nabla v|^{\gamma + q + 2},
		\end{aligned}
	\end{equation}
	with $(\frac{2}{q}, \frac{2}{2 - q})$. We only need that
	\begin{equation}\label{M_large_1}
		\begin{aligned}
			&\mathbb K_5 - PJ > 0,\\
		\end{aligned}
	\end{equation}
	and
	\begin{equation}\label{M_large_2}
		\begin{aligned}
			&\mathbb K_6 - P\cdot\frac{q}{2}\Big(J\cdot\frac{2}{2 - q}\Big)^{-\frac{2 - q}{q}}M^{-\frac{2}{q}} > 0.
		\end{aligned}
	\end{equation}
	\begin{claim}\label{claim7}
		When $\frac{n + 2}{n - 2} - \frac{1}{n^2} \leq p \leq \frac{n + 2}{n - 2}$ and $n \geq 7$,
we have
\begin{equation*}
	1 + \frac{1.9}{n} < q < 1 + \frac{2}{n}.
\end{equation*}
	\end{claim}
	
	Thus by \eqref{Theorem_equ1} we get that
	\begin{align*}
		\Rightarrow &\mathbb K_5=- I\left(U_{0}, \frac{n + 2}{n - 2} - \frac{1}{n^2}\right) \\
				&>  -\Bigg[ -2.8 + 0.051(2 - q) + n(q - 1)   - 0.5\sqrt{2}q \Bigg] \frac{1}{n} \\
			&- \Bigg[ 6 + 4.34(2 - q) + nq(1 - q) + 1.41q \Bigg]\frac{1}{n^2}\\
			&> - \Bigg[ -2.8 + 0.051(2 - q) + n(q - 1)   - 0.5\sqrt{2}q \Bigg] \frac{1}{n} \\
			&- \Bigg[ 6 + 4.32(2 - q) - 1.9\left(1 + \frac{1.9}{n}\right) + 1.41q \Bigg]\frac{1}{n^2}.\\
	\end{align*}
	\begin{equation*}
		\begin{aligned}
			\Rightarrow &\mathbb K_5 > -\left(n - 0.758 - \frac{2.93}{n}\right)\frac{q}{n} + \frac{n + 2.698 }{n} - \frac{12.78}{n^2},\\
		\end{aligned}
	\end{equation*}
	Using \eqref{M_large_1}, when $T$ tends to zero, we can choose $J$ such that $\mathbb K_5 - PJ$ tends to zero . Then we get
	\begin{equation}
		\begin{aligned}
			PJ > -\left(n - 0.758 - \frac{2.93}{n}\right)\frac{q}{n} + \frac{n + 2.698 }{n} - \frac{12.78}{n^2}.
		\end{aligned}
	\end{equation}
	By \eqref{M_large_2}, we get
	\begin{align*}
		&P\cdot\frac{q}{2}\Big(J\cdot\frac{2}{2 - q}\Big)^{-\frac{2 - q}{q}}M^{-\frac{2}{q}} < -\frac{q}{\gamma + q} + U_0,\\
		\Leftrightarrow & P\cdot\frac{q}{2}\Big(J\cdot\frac{2}{2 - q}\Big)^{-\frac{2 - q}{q}}\left( -\frac{q}{\gamma + q} + U_0\right)^{-1} <  M^{\frac{2}{q}},\\
		\Leftrightarrow &M_2 := P^{\frac{q}{2}} \cdot\left(\frac{q}{2}\right)^{\frac{q}{2}}\Big(J\cdot\frac{2}{2 - q}\Big)^{-\frac{2 - q}{2}}\left( -\frac{q}{\gamma + q} + U_0\right)^{-\frac{q}{2}} <  M.\\
	\end{align*}
	So far, we have already proved that for any $M > M_2$, the condition \eqref{condition08240930} is satisfied. Finally we will show that
	\begin{align*}
		&M_2:=P^{\frac{q}{2}} \cdot\left(\frac{q}{2}\right)^{\frac{q}{2}}\Big(J\cdot\frac{2}{2 - q}\Big)^{-\frac{2 - q}{2}}\left( -\frac{q}{n - 4 + q} + U_0\right)^{-\frac{q}{2}}\\
		& <  (p + 1)\Bigg[ \frac{(n - 1)^2q}{4n} - 1 \Bigg]^{- 1} \Bigg[ \frac{n}{q} - \frac{n(n - 1)}{n + 2} \Bigg]^{\frac{q}{2}} = : M_1.
	\end{align*}
	\begin{align*}
		M_1 &=  (p + 1)\Bigg[ \frac{(n - 1)^2q}{4n} - 1 \Bigg]^{- 1} \Bigg[ \frac{n}{q} - \frac{n(n - 1)}{n + 2} \Bigg]^{\frac{q}{2}}\\
		&>  (p + 1)\Bigg[ \frac{(n - 1)^2q}{4n} - 1 \Bigg]^{- 1} \Bigg[ \frac{n^2}{n + 2} - \frac{n(n - 1)}{n + 2} \Bigg]^{\frac{q}{2}}\\
		&>  \left(\frac{n + 2}{n - 2} - \frac{1}{n^2} + 1\right)\Bigg[ \frac{(n - 1)^2(n + 2)}{4n^2} - 1 \Bigg]^{- 1} \Bigg( \frac{n}{n + 2} \Bigg)^{\frac{n + 2}{2n}}\\
		&>  \left(\frac{n + 1}{n - 2} + 1\right)\Bigg[ \frac{(n - 1)^2(n + 2)}{4n^2} - 1 \Bigg]^{- 1} \Bigg( \frac{7}{9} \Bigg)^{\frac{n + 2}{2n}}\\
		&> 2\Bigg( \frac{7}{9} \Bigg)^{\frac{9}{14}} \Bigg[ \frac{(n - 1)^2(n + 2)}{4n^2} - 1 \Bigg]^{- 1}\\
		&> 1.7\Bigg[ \frac{(n - 1)^2(n + 2)}{4n^2} - 1 \Bigg]^{- 1}.\\
	\end{align*}

	\begin{claim}\label{claim8}
		When $n \geq 7$, we have
		\begin{align*}
		U_0 &> (4 - 2\sqrt{2}) \left(\frac{1}{n} + \frac{7}{n^2} \right).\\
	\end{align*}
	\end{claim}
	Besides we know that the function $y = x^{-\frac{1}{x}}$ is decreasing when $x > 1$. So we get that
	\begin{align*}
		\left(\frac{2}{2 - q}\right)^{-\frac{2 - q}{2}} < \left(\frac{2}{2 - 1}\right)^{-\frac{2 - 1}{2}} = \frac{\sqrt{2}}{2}.
	\end{align*}
	Hence we obtain that:	
	\begin{align*}
		M_2 &= P^{\frac{q}{2}} \cdot\left(\frac{q}{2}\right)^{\frac{q}{2}}\Big(J\cdot\frac{2}{2 - q}\Big)^{-\frac{2 - q}{2}}\left( -\frac{q}{n - 4 + q} + U_0\right)^{-\frac{q}{2}}\\
		&= P\cdot\left(\frac{q}{2}\right)^{\frac{q}{2}}\Big(PJ\cdot\frac{2}{2 - q}\Big)^{-\frac{2 - q}{2}}\left( -\frac{q}{n - 4 + q} + U_0\right)^{-\frac{q}{2}}\\
		&< \frac{1}{n^2}\Bigg[-\left(n - 0.758 - \frac{2.93}{n}\right)\frac{q}{n} + \frac{n + 2.698 }{n} - \frac{12.78}{n^2} \Bigg]^{-\frac{2 - q}{2}}\\
		&\cdot \left(\frac{2}{2 - q}\right)^{-\frac{2 - q}{2}}\cdot\left( -\frac{1 + \frac{2}{n}}{n - 3 + \frac{2}{n}} + \frac{4 - 2\sqrt{2}}{n} + \frac{28 - 14\sqrt{2}}{n^2}\right)^{-\frac{q}{2}}\\
		&< \frac{1}{n^2} \Bigg[-\left(n - 0.758 - \frac{2.93}{n}\right)\frac{n + 2}{n^2} + \frac{n + 2.698 }{n} - \frac{12.78}{n^2} \Bigg]^{-\frac{2 - q}{2}}\\
		&\cdot \frac{\sqrt{2}}{2}\left( -\frac{1 + \frac{2}{n}}{n - 3} + \frac{4 - 2\sqrt{2}}{n} + \frac{28 - 14\sqrt{2}}{n^2}\right)^{-\frac{q}{2}}\\
		&< \frac{1}{n^q} \Bigg( 1.456n - 8.334 \Bigg)^{-\frac{2 - q}{2}} \cdot \frac{\sqrt{2}}{2}\left( -\frac{1 + \frac{2}{n}}{n - 3} + \frac{4 - 2\sqrt{2}}{n} + \frac{28 - 14\sqrt{2}}{n^2}\right)^{-\frac{q}{2}}\\
		&< n^{- 1 - \frac{1.9}{n}} \Bigg( 1.456n - 8.334 \Bigg)^{-\frac{n - 2}{2n}} \cdot \frac{\sqrt{2}}{2}\left( -\frac{1 + \frac{2}{n}}{n - 3} + \frac{4 - 2\sqrt{2}}{n} + \frac{28 - 14\sqrt{2}}{n^2}\right)^{-\frac{q}{2}}\\
	\end{align*}
	When $n = 7$, we have
	\begin{align*}
		M_1 > 2.6,
	\end{align*}
	and
	\begin{align*}
		M_2 &< n^{- 1 - \frac{1.9}{n}} \Bigg( 1.456n - 8.334 \Bigg)^{-\frac{n - 2}{2n}} \cdot \frac{\sqrt{2}}{2}\left( -\frac{1 + \frac{2}{n}}{n - 3} + \frac{4 - 2\sqrt{2}}{n} + \frac{28 - 14\sqrt{2}}{n^2}\right)^{-\frac{n + 2}{2n}}\\
		&< 0.8.
	\end{align*}
	When $n \geq 8$, we have
	\begin{align*}
		 &-\frac{1 + \frac{2}{n}}{n - 3} + \frac{1}{n} + \frac{28 - 14\sqrt{2}}{n^2}\\
		 & > -\frac{1 + \frac{2}{n}}{n - 3} + \frac{1}{n} + \frac{8.2}{n^2}\\
		 &= \frac{-n^2 - 2n + n^2 - 3n + 8.2n - 24.6}{n^2(n - 3)}\\
		 &= \frac{ 3.2n - 24.6}{n^2(n - 3)} > 0.
	\end{align*}	
	Thus we deduce that
	\begin{align*}
		M_2 &< n^{- 1 - \frac{1.9}{n}} \Bigg( 1.456n - 8.334 \Bigg)^{-\frac{n - 2}{2n}} \cdot \frac{\sqrt{2}}{2}\cdot\left(\frac{3 - 2\sqrt{2}}{n}\right)^{-\frac{q}{2}}\\
		&< n^{- 1 - \frac{1.9}{n}}\Bigg(1.456n - 8.334 \Bigg)^{-\frac{n - 2}{2n}} \cdot \frac{\sqrt{2}}{2}\cdot\left( \frac{3 - 2\sqrt{2}}{n}\right)^{-\frac{n + 2}{2n}}\\
		&= n^{- 1}\Bigg( 1.456n - 8.334 \Bigg)^{-\frac{1}{2}} \cdot \frac{\sqrt{2}}{2}\cdot\left( \frac{3 - 2\sqrt{2}}{n}\right)^{-\frac{1}{2}} \left[ \frac{1.456n - 8.334}{(3 - 2\sqrt{2})n}\right]^{\frac{1}{n}}n^{\frac{0.1}{n}} .\\
	\end{align*}	
	Since the function $ y = n^{\frac{1}{n}}$ is decreasing when $n \geq e$, we get that
	\begin{align*}
		M_2 &<  n^{- 1}\Bigg( 1.456n - 8.334 \Bigg)^{-\frac{1}{2}} \cdot \frac{\sqrt{2}}{2}\cdot\left( \frac{3 - 2\sqrt{2}}{n}\right)^{-\frac{1}{2}} \left[ \frac{1.456}{3 - 2\sqrt{2}}\right]^{\frac{1}{8}}8^{\frac{0.1}{8}} \\
		&< n^{- 1}\Bigg( 1.456n - 8.334 \Bigg)^{-\frac{1}{2}} \left( \frac{3 - 2\sqrt{2}}{n}\right)^{-\frac{1}{2}} \\
		&< n^{-1}\left(0.24 - \frac{1.43}{n}\right)^{-\frac{1}{2}}. \\
	\end{align*}		
	Therefore, if $n \geq 8$, we need to show that	
	\begin{align*}
		&M_2 < M_1,\\
		\Leftarrow&n^{-1}\left(0.24 - \frac{1.43}{n}\right)^{-\frac{1}{2}}  < 1.7\Bigg[ \frac{(n - 1)^2(n + 2)}{4n^2} - 1 \Bigg]^{- 1},\\
		\Leftrightarrow &\Bigg[ \frac{(n - 1)^2(n + 2)}{4n^2} - 1 \Bigg] < 1.7n\left(0.24 - \frac{1.43}{n}\right)^{\frac{1}{2}},\\
		\Leftarrow &\Bigg[ \frac{(n - 1)^2(n + 2)}{4n^2} - 1 \Bigg] < 1.7n\left(0.24 - \frac{1.43}{8}\right)^{\frac{1}{2}},\\
		\Leftarrow & \frac{(n - 1)^2(n + 2)}{4n^2} - 1 < 0.4n,\\
		\Leftarrow &0.6n^3 + 4n^2 + 3n - 2> 0.
	\end{align*}	
		
	\end{proof}

	\begin{appendices}
		\section{Proof of claims in Section \ref{Proof of Thm3}}

		\begin{proof}[Proof of claim \ref{claim10}]
Denote
\begin{align*}
    \qquad D_1&=\frac{\sqrt{(n-2)(2n-6)}-2}{n-4}\\
    \qquad D_2&=\frac{2\sqrt{(n-2)(2n-6)}+n^2-5n+8}{(n-1)^2}
\end{align*}
then
\begin{align*}
    & H=4(n-2)D_1D_2 \left(-n+3-\frac{4}{n}\right)\times\left\{\frac{1}{n^2}-\frac{n+2}{n-2}+\frac{n - 2 - \sqrt{(n-2)(2n-6)}}{n}\right\}\\
    &-\left\{-(n+2)+\frac{n-2}{n^2}-\frac{n-2}{n}\left (2\sqrt{(n-2)(2n-6)}-n+4 \right ) \right\}^2\\
    &=4(n-2)D_1D_2 \left( n-3+\frac{4}{n} \right)\frac{4n^2-n+2+n(n-2)\left( 2+\sqrt{(n-2)(2n-6)}\right)}{n^2(n-2)}\\
    &-\left\{ \frac{1}{n^2}\left[8n^2-9n+2+2n(n-2) \sqrt{(n-2)(2n-6)}\right] \right\}^2.
\end{align*}
$\Rightarrow$
\begin{align*}
    n^4H=&4n^2D_1D_2\left(n-3+\frac{4}{n}\right)\left[6n^2-5n+2+n(n-2)\sqrt{(n-2)(2n-6)} \right]\\
    &-\left\{
        8n^2-9n+2+2n(n-2)\sqrt{(n-2)(2n-6)}
    \right\}^2\\
    D_1\times D_2=&\frac{\left(\sqrt{(n-2)(2n-6)}-2\right)\left(2\sqrt{(n-2)(2n-6)}+n^2-5n+8\right)}{(n-1)^2(n-4)}\\
    =&\frac{2(n^2-5n+4)+(n-1)(n-4)\sqrt{(n-2)(2n-6)}}{(n-1)^2(n-4)}\\
    =&\frac{2+\sqrt{(n-2)(2n-6)}}{n-1}.
\end{align*}
Let $\Delta^\prime=(n-2)(2n-6) $
\begin{align*}
    \Rightarrow n^4H=&4n^2\frac{2+\sqrt{\Delta^\prime }}{n-1}\frac{n^2-3n+4}{n} \left[ 6n^2-5n+2+n(n-2)\sqrt{\Delta^\prime }\right]\\
    &-\left[8n^2-9n+2+2n(n-2)\sqrt{\Delta^\prime } \right]^2,
 \end{align*}
 so
  \begin{align*}
    (n-1)n^4H=&4n(n^2-3n+4)(2+\sqrt{\Delta^\prime })\left[ 6n^2-5n+2+n(n-2)\sqrt{\Delta^\prime }\right]\\
    &-(n-1)\left[8n^2-9n+2+2n(n-2)\sqrt{\Delta^\prime } \right]^2\\
    =&8n(n^2-3n+4)(6n^2-5n+2)\\
    &+4n^2(n-2)^2(n^2-3n+4)(2n-6)-4n^2(n-2)^3(2n-6)(n-1)\\
    &-(8n^2-9n+2)^2(n-1)+\sqrt{\Delta^\prime }\cdot 8n(8n^2-9n+2)\\
     = &8n(n^2-3n+4)(6n^2-5n+2)+8n^2(n-2)^2(2n-6)\\
    &-(8n^2-9n+2)^2(n-1) +\sqrt{\Delta^\prime }\cdot 8n(8n^2-9n+2)\\
	=& -88n^4 + 327n^3 -251n^2 + 24n + 4+\sqrt{\Delta^\prime }\cdot 8n(8n^2-9n+2).\\
\end{align*}
	Since we know when $n \geq 7$,
	\begin{align*}
		\sqrt{\Delta '} &> \sqrt{2}\left[n - 2.5 - \frac{1}{6(n - 2)}\right]\\
		&> \sqrt{2}\left(n - 2.5 - \frac{1}{30}\right)\\
		&= \sqrt{2}\left(n - \frac{38}{15}\right),\\
	\end{align*}
	then we get that
	\begin{align*}
		(n-1)n^4H &> -88n^4 + 327n^3 -251n^2 + 24n + 4 + 8\sqrt{2}n\left(n - \frac{38}{15}\right)\cdot (8n^2-9n+2)\\
		&> 2.5n^4 - 4.2n^3 + 29n^2 - 34n + 4 > 0.
	\end{align*}		
			
		\end{proof}

		\begin{proof}[Proof of claim \ref{claim1} ]
			In fact, we know
			\begin{align*}
				&\left(1 + \frac{2}{n}\right)\left(n - 3 + \frac{2}{n}\right)(n - 2)^2(n - 1)\\
				&= \left(n - 1 - \frac{4}{n} + \frac{4}{n^2}\right)(n - 2)^2(n - 1)\\
				&< \left(n - 1 - \frac{3}{n}\right)(n - 2)^2(n - 1)\\
				&= \left(n^2 - 3n\right)(n - 2)(n - 1),\\
			\end{align*}
			and
			\begin{align*}
				 &2\sqrt{(n - 2)(2n - 6)} + n^2 - 5n + 8 \\
				 &> 2\sqrt{2}(n - 3) + n^2 - 5n + 8 \\
				 &= n^2 - 2.3n + (2\sqrt{2} - 2.7)n + 8 - 6\sqrt{2}\\
				 &> n^2 - 2.3n.
			\end{align*}
			Then we get that
			\begin{align*}
				\Delta &> 1 - \frac{(1 + \frac{2}{n})(n - 3 + \frac{2}{n})(n - 2)^2(n - 1)}{ (\frac{n}{n - 1} + n - 4 )(2n - 6)\Big[ 2\sqrt{(n - 2)(2n - 6)} + n^2 - 5n + 8\Big] }\\
				&> 1 - \frac{\left(n^2 - 3n\right)(n - 2)(n - 1)}{  (n - 3 + \frac{1}{n - 1})(2n - 6)(n^2 - 2.3n) }\\
				&= 1 - \frac{(n - 2)(n - 1)}{  2(n - 3 + \frac{1}{n - 1})(n - 2.3) }\\
				&= 1 - \frac{(n - 2)(n - 1)}{  2(n^2 - 5.3n + 7.9 - \frac{1.3}{n - 1}) }\\
				&> 1 - \frac{(n - 2)(n - 1)}{  2(n^2 - 5.3n + 7.6) }.\\
			\end{align*}
			Then
			\begin{align*}
				&\Delta > 0.5 - \frac{2}{n} + \frac{2}{n^2},\\
				\Leftrightarrow &1 - \frac{(n - 2)(n - 1)}{  2(n^2 - 5.3n + 7.6) } > 0.5 - \frac{2}{n} + \frac{2}{n^2},\\
				\Leftrightarrow &\frac{(n - 2)(n - 1)}{(n^2 - 5.3n + 7.6) } < \frac{n^2 + 4n - 4}{n^2},\\
				\Leftrightarrow h_1(n):= &-1.7n^3 + 19.6n^2 - 51.6n + 30.4 < 0
			\end{align*}
			Since $n \geq 9$, we know
			\begin{align*}
				h'(n) = -5.1n^2 + 39.2n - 51.6 \leq h'(9) < 0.
			\end{align*}
			So we deduce that
			\begin{align*}
				h(n) < h(9) < 0.
			\end{align*}
		\end{proof}

		\begin{proof}[Proof of claim \ref{claim2} ]
			\begin{align*}
				LHS &= \left( \frac{\sqrt{2} - 1}{2} + \frac{1}{n} \right) \cdot \frac{n - 4}{n - \frac{5}{2} - \frac{1}{6(n - 2)}  - \sqrt{2} }\left(1 + \frac{4}{n - 2} - \frac{1}{n^2}\right)\\
				&= \left( \frac{\sqrt{2} - 1}{2} + \frac{1}{n} + \frac{2\sqrt{2} - 2}{n - 2} + \frac{4}{n(n - 2)} - \frac{\sqrt{2} - 1}{2n^2} - \frac{1}{n^3} \right) \cdot \frac{n - 4}{n - \frac{5}{2} - \frac{1}{6(n - 2)}  - \sqrt{2} }\\
				&< \left( \frac{\sqrt{2} - 1}{2} + \frac{1}{n} + \frac{2\sqrt{2} - 2}{n - 2} + \frac{4}{n(n - 2)} - \frac{\sqrt{2} - 1}{2n^2} - \frac{1}{n^3} \right) \cdot \frac{n - 4}{n - \frac{5}{2} - \frac{1}{30}  - \sqrt{2} }\\
				&<  \frac{\sqrt{2} - 1}{2} + \frac{1}{n} + \frac{2\sqrt{2} - 2}{n} + \frac{2\sqrt{2} - 2}{n - 2} - \frac{2\sqrt{2} - 2}{n} + \frac{4}{n(n - 2)} - \frac{\sqrt{2} - 1}{2n^2} - \frac{1}{n^3} \\
			\end{align*}
			Thus we get that
			\begin{align*}
				LHS &< \frac{\sqrt{2} - 1}{2} + \frac{2\sqrt{2} - 1}{n} + \frac{4\sqrt{2}}{n(n - 2)} \\
				&\leq \frac{\sqrt{2} - 1}{2} + \frac{2\sqrt{2} - 1}{n} + \frac{4\sqrt{2}}{5n}  \\
				&< \frac{\sqrt{2} - 1}{2} + \frac{3}{n}.
			\end{align*}
		\end{proof}

		\begin{proof}[Proof of claim \ref{claim3} ]
		\begin{align*}
			LHS &= 2\sqrt{2} \left(n - 3 + \frac{1}{n - 1}\right)\frac{2\sqrt{2}\Big( n - \frac{38}{15} \Big) + n^2 - 5n + 8}{(n - 1)(n - 2)^2(n - 4)}\left( n - \frac{38}{15}  -  \frac{1.5}{\sqrt{2}}\right)\\
			&> 2\sqrt{2} \left(n - 3 + \frac{1}{n - 1}\right)\frac{n^2 -2.2n + 0.83}{(n - 1)(n - 2)^2(n - 4)}\left( n - 3.6\right)\\
			&> 2\sqrt{2} \left(n - 3 + \frac{1}{n}\right)\frac{n^2 -2.2n + 0.83}{(n - 1)(n - 2)^2(n - 4)}\left( n - 3.6\right)\\
			&> 2\sqrt{2} \left(n^3 - 5.2n^2 + 8.43n - 4.69\right)\frac{n - 3.6}{(n - 1)(n - 2)^2(n - 4)}\\
			&= 2\sqrt{2} \cdot \frac{n^3 - 5.2n^2 + 8.43n - 4.69}{n^3 - 8n^2 + 20n - 16}\cdot \frac{n - 3.6}{n - 1}.\\
		\end{align*}
		So we only need to show that
		\begin{align*}
			&\frac{n^3 - 5.2n^2 + 8.43n - 4.69}{n^3 - 8n^2 + 20n - 16}\cdot \frac{n - 3.6}{n - 1} > \frac{n + 0.2}{n},\\
			\Leftrightarrow & n(n - 3.6)(n^3 - 5.2n^2 + 8.43n - 4.69) - (n - 1)(n + 0.2)(n^3 - 8n^2 + 20n - 16) > 0.
		\end{align*}
		Actually, we know
		\begin{align*}
			LHS &= n^5 - 8.8n^4 + 27.15n^3 - 35.038n^2 + 16.884n - (n^5 - 8.8n^4 + 26.2n^3 - 30.4n^2 + 8.8n + 3.2)\\
			&= 0.95n^3 - 4.638n^2 + 8.084n - 3.2\\
			&> 0.95\cdot 7^3 - 4.638\cdot 7^2 + 8.084\cdot7 - 3.2 > 0.\\
		\end{align*}
		\end{proof}

		\begin{proof}[Proof of claim \ref{claim4} ]
		When $n =7$, we have
		\begin{align*}
			LHS < 0.96 < RHS.
		\end{align*}
		When $n = 8$, we have
		\begin{align*}
			LHS < 0.9 < RHS.
		\end{align*}
		When $n \geq 9$, we have
		\begin{align*}
			LHS &< \frac{2\sqrt{2}\Big[ n - \frac{5}{2} - \frac{1}{8(n - 2)}\Big] - n + 4}{n - 2}\cdot \left( \frac{2 - \sqrt{2}}{2} + \frac{\sqrt{2}}{n} \right) \\
			&< \frac{2\sqrt{2}\Big[ n - \frac{5}{2} - \frac{1}{8n}\Big] - n + 4}{n - 2}\cdot \left( \frac{2 - \sqrt{2}}{2} + \frac{\sqrt{2}}{n} \right) \\
			&= \left[ (2\sqrt{2} - 1)n + 4 - 5\sqrt{2} - \frac{\sqrt{2}}{4n}\right]\cdot \frac{1}{n} \left(1 + \frac{2}{n} + \frac{4}{n^2}\cdot\frac{n}{n - 2}\right) \cdot \left( \frac{2 - \sqrt{2}}{2} + \frac{\sqrt{2}}{n} \right) \\
			&\leq \left( 2\sqrt{2} - 1 + \frac{ 4 - 5\sqrt{2}}{n} - \frac{\sqrt{2}}{4n^2}\right)\cdot  \left(1 + \frac{2}{n} + \frac{4}{n^2}\cdot\frac{9}{7}\right) \cdot \left( \frac{2 - \sqrt{2}}{2} + \frac{\sqrt{2}}{n} \right) \\
		\end{align*}
		\begin{align*}
			LHS &=  \Bigg[ 2\sqrt{2} - 1 + \frac{ 2 - \sqrt{2}}{n} + \frac{72\sqrt{2} - 36}{7n^2} + \frac{8 - 10\sqrt{2}}{n^2} - \frac{\sqrt{2}}{4n^2} + \frac{144 - 180\sqrt{2}}{7n^3} - \frac{\sqrt{2}}{2n^3} \\
			&- \frac{9\sqrt{2}}{7n^4}\Bigg] \cdot \left( \frac{2 - \sqrt{2}}{2} + \frac{\sqrt{2}}{n} \right)\\
			&< \left( 2\sqrt{2} - 1 + \frac{2 - \sqrt{2}}{n} + \frac{2.91}{n^2} - \frac{16.5}{n^3} \right) \cdot \left( \frac{2 - \sqrt{2}}{2} + \frac{\sqrt{2}}{n} \right) \\
			&= (2\sqrt{2} - 1)\frac{2 - \sqrt{2}}{2} + \frac{4 - \sqrt{2}}{n} + \frac{3 - 2\sqrt{2}}{n} + \frac{2\sqrt{2} - 2}{n^2} + \frac{2.91(2 - \sqrt{2})}{2n^2} + \frac{2.91\sqrt{2}}{n^3}\\
			&- \frac{16.5(2 - \sqrt{2})}{2n^3} - \frac{16.5\sqrt{2}}{n^4}\\
			&= (2\sqrt{2} - 1)\frac{2 - \sqrt{2}}{2} + \frac{7 - 3\sqrt{2}}{n} + \frac{0.91 + 0.545\sqrt{2}}{n^2} + \frac{11.16\sqrt{2} - 16.5}{n^3} - \frac{16.5\sqrt{2}}{n^4}\\
			&< (2\sqrt{2} - 1)\frac{2 - \sqrt{2}}{2} + \frac{7 - 3\sqrt{2}}{n} + \frac{0.91 + 0.545\sqrt{2}}{n^2} \\
			&< (2\sqrt{2} - 1)\frac{2 - \sqrt{2}}{2} + \frac{2.76}{n} + \frac{1.7}{n^2}.
		\end{align*}

		\end{proof}

		\begin{proof}[Proof of claim \ref{claim5} ]
		When $n =7$,
		\begin{align*}
			LHS > 0.05 > RHS.
		\end{align*}
		When $n = 8$,
		\begin{align*}
			LHS > 0.04 > RHS.
		\end{align*}
		When $n \geq 9$
		\begin{align*}
			LHS &>  \frac{1}{n - 2}\cdot \frac{ \frac{2 - \sqrt{2}}{2} n  - \frac{53}{21}  + \sqrt{2}}{n - \frac{5}{2} - \sqrt{2}}\cdot\frac{2\sqrt{2}\Big[ n - \frac{5}{2} - \frac{1}{6(n - 2)}\Big] - n + 4}{n - 2}\left( \frac{2 - \sqrt{2}}{2} + \frac{\sqrt{2}}{n} \right)\\
			&\geq  \frac{1}{n - 2}\cdot \frac{ \frac{2 - \sqrt{2}}{2} n  - \frac{53}{21}  + \sqrt{2}}{n - \frac{5}{2} - \sqrt{2}}\cdot\frac{2\sqrt{2}\Big[ n - \frac{5}{2} - \frac{1}{42}\Big] - n + 4}{n - 2}\left( \frac{2 - \sqrt{2}}{2} + \frac{\sqrt{2}}{n} \right)\\
			&> \frac{2 - \sqrt{2}}{2} (2\sqrt{2} - 1)\cdot\frac{2 - \sqrt{2}}{2} \cdot \frac{(n - 3.8)(n - 1.72)}{(n - 2)^2(n - 3.9)}\cdot  \frac{n + 4.8}{n} \\
			&= \frac{8\sqrt{2} - 11}{2}\cdot \frac{(n - 3.8)(n - 1.72)}{(n - 2)^2(n - 3.9)}\cdot  \frac{n + 4.8}{n} \\
		\end{align*}
		So we only need to check that
		\begin{align*}
			&\frac{(n - 3.8)(n - 1.72)(n + 4.8)}{(n - 2)^2(n - 3.9)} > \frac{n + 7}{n},\\
			\Leftrightarrow &n(n - 3.8)(n - 1.72)(n + 4.8) - (n + 7)(n - 2)^2(n - 3.9) > 0,\\
			\Leftarrow &n(n - 1.72)(n + 4.8) - (n + 7)(n - 2)^2 = 0.08n^2 + 15.744n - 28 > 0.\\
		\end{align*}

		\end{proof}

		\begin{proof}[Proof of claim \ref{claim6} ]
		\begin{align*}
				\frac{ n - 4 + \frac{4}{n}}{n - 4 + q} &= \left( 1 - \frac{4}{n} + \frac{4}{n^2} \right)\sum_{k = 0}^{\infty}\left( \frac{4 - q}{n}\right)^k\\
				&< \left( 1 - \frac{4}{n} + \frac{4}{n^2} \right)\left(1 + \frac{4 - q}{n} + \frac{(4 - q)^2}{n^2} \cdot \frac{1}{1 - \frac{4 - q}{n}} \right)\\
				&\leq \left( 1 - \frac{4}{n} + \frac{4}{n^2} \right)\left(1 + \frac{4 - q}{n} + \frac{(4 - q)^2}{n^2} \cdot \frac{1}{1 - \frac{4 - 1}{9}} \right)\\
				&= \left( 1 - \frac{4}{n} + \frac{4}{n^2} \right)\left(1 + \frac{4 - q}{n} + \frac{1.5(4 - q)^2}{n^2} \right)\\
				&= 1 - \frac{q}{n} + \frac{1.5(4 - q)^2 - 16 + 4q + 4}{n^2} + \frac{4 - q}{n^3}\left[-6(4 - q) + 4 + \frac{6(4 - q)}{n}\right]\\
				&<  1 - \frac{q}{n} + \frac{1.5q^2 - 8q + 12}{n^2}
			\end{align*}

		\end{proof}

	\begin{proof}[Proof of claim \ref{claim 9}]
			Recall that we have already known \eqref{n>7} :
				\begin{equation*}
			\begin{aligned}
				\Rightarrow &I\left(U_{0}, \frac{n + 2}{n - 2} - \frac{1}{n^2}\right) \\
			&\leq  -\frac{\sqrt{2}}{2} - \frac{\sqrt{2}}{n} + \frac{\sqrt{2}}{2n^2} + \frac{2 - q}{n - 3}\Bigg( \frac{\sqrt{2} - 1}{2} + \frac{3}{n}  \Bigg)  \\
			&- \frac{ n - 4 + \frac{4}{n}}{n - 4 + q}\Bigg[ 1 - q - 0.5\sqrt{2} + \frac{0.84 + 0.4\sqrt{2}}{n} - \frac{0.9}{n^2} + (2 - q)\frac{8\sqrt{2} - 11}{2}\left( \frac{1}{n} + \frac{7}{n^2}\right)   \Bigg].\\
			\end{aligned}
		\end{equation*}
			When $n = 7$ and $1 < q \leq \frac{9}{7}$ ,
		\begin{align*}
			&I\left(U_{0}, \frac{n + 2}{n - 2} - \frac{1}{n^2}\right) \\
			&< -0.89 + 0.16(2 - q) - \frac{25}{7(q + 3)}\Bigg[ -q + 0.47 + 0.044(2 - q) \Bigg]\\
			&<  -0.89 + 0.16(2 - q) - \frac{25}{28}\Bigg[ -q + 0.47 + 0.044(2 - q) \Bigg]\\
			&< 0.773 q - 1.068.
		\end{align*}
		and
		\begin{align*}
			&\Bigg[ -2.8 + 0.051(2 - q) + n(q - 1)   - 0.5\sqrt{2}q \Bigg] \frac{1}{n}\\
			&+ \Bigg[ 6 + 4.34(2 - q)  + nq(1 - q)+ 1.41q \Bigg] \frac{1}{n^2}\\
			& > -\frac{1}{7}q^2 + 0.974q - 1.086 > 0.773 q - 1.068.
		\end{align*}
		When $n = 8$ and $1 < q \leq \frac{5}{4}$,
		\begin{align*}
			&I\left(U_{0}, \frac{n + 2}{n - 2} - \frac{1}{n^2}\right) \\
			&< -0.87 + 0.12(2 - q) - \frac{4.5}{4 + q}\Bigg[ -q + 0.45 + 0.036(2 - q) \Bigg]\\
			&< -0.87 + 0.12(2 - q) - \frac{4.5}{5}\Bigg[ -q + 0.45 + 0.036(2 - q) \Bigg]\\
			&< 0.8124q - 1.0998,
		\end{align*}
		and
		\begin{align*}
			&\Bigg[ -2.8 + 0.051(2 - q) + n(q - 1)   - 0.5\sqrt{2}q \Bigg] \frac{1}{n}\\
			&+ \Bigg[ 6 + 4.34(2 - q)  + nq(1 - q)+ 1.41q \Bigg] \frac{1}{n^2}\\
			& > -\frac{1}{8}q^2 + 0.98q -1.108 > 0.8124q - 1.0998.
		\end{align*}
		\end{proof}

	\begin{proof}[Proof of claim \ref{claim7}]
	If $q > 1 + \frac{1.9}{n}$, then
	\begin{align*}
		p = \frac{q}{2 - q} > \frac{n + 1.9}{n - 1.9}.
	\end{align*}
	So we only need to check that
	\begin{align*}
		&\frac{n + 2}{n - 2} - \frac{1}{(n - 2)^2} > \frac{n + 1.9}{n - 1.9},\\
		\Leftrightarrow & \frac{n^2 - 5}{(n - 2)^2} - \frac{n + 1.9}{n - 1.9} > 0,\\
		\Leftrightarrow &(n^2 - 5)(n - 1.9) - (n + 1.9)(n - 2)^2 > 0,\\
		\Leftrightarrow &0.2n^2 - 1.4n + 1.9 > 0.
	\end{align*}
	The last inequality follows from $n \geq 7$.

	\end{proof}

	\begin{proof}[Proof of claim \ref{claim8} ]
		
	By claim \ref{claim3} we know
	\begin{align*}
		U_0 &= \frac{2B_0 (1 + \gamma S + qS)^2}{(\gamma + q)^2}\left( \frac{2 - \sqrt{2}}{2} + \frac{\sqrt{2}}{n} \right)\\
		&= \frac{2B_0(1 + \gamma S + qS)}{\gamma + q} \left(\frac{1}{\gamma + q} + S \right) \left( \frac{2 - \sqrt{2}}{2} + \frac{\sqrt{2}}{n} \right)\\
		&> \left(2\sqrt{2} + \frac{0.4\sqrt{2}}{n}\right)\frac{\sqrt{2}(n - \frac{38}{15}) - 1.5}{(n - 2)(n - 4)}\left( \frac{2 - \sqrt{2}}{2} + \frac{\sqrt{2}}{n} \right)\\
		&> \frac{2 - \sqrt{2}}{2} \left(4 + \frac{0.8}{n}\right)\frac{n - 3.6}{(n - 2)(n - 4)}\left( 1 + \frac{4.8}{n} \right)\\
		&=(4 - 2\sqrt{2}) \left(1 + \frac{0.2}{n}\right)\frac{n - 3.6}{(n - 2)(n - 4)}\left( 1 + \frac{4.8}{n} \right).\\
	\end{align*}
		We only need to show that
		\begin{align*}
			&\frac{(n + 0.2)(n - 3.6)(n + 4.8)}{n^2(n - 2)(n - 4)} > \frac{n + 7}{n^2},\\
			\Leftrightarrow &(n + 0.2)(n - 3.6)(n + 4.8) - (n + 7)(n - 2)(n - 4) > 0.
		\end{align*}
	\begin{align*}
		LHS &= 0.4n^2 + 16.96n - 59.456 > 0.
	\end{align*}
	\end{proof}

	\end{appendices}

\bibliography{241212MWZ}
\bibliographystyle{plain}

\end{document}